\documentclass[11pt,,leqno,twoside]{article} 
\usepackage{bbm}
\usepackage{mathrsfs}
\usepackage{amssymb}
\usepackage{amsmath}
\usepackage{amsthm}
\usepackage{amsfonts}
\usepackage{color}
\usepackage{graphicx}
\usepackage[active]{srcltx}
\usepackage{CJK}

\usepackage[cp1252]{inputenc}

\usepackage{mathrsfs}
\usepackage{graphicx}

\usepackage[active]{srcltx}

\allowdisplaybreaks

\usepackage{titletoc}
\titlecontents{section}[0pt]{\addvspace{2pt}\filright}
              {\contentspush{\thecontentslabel\ }}
              {}{\titlerule*[8pt]{.}\contentspage}

\def\rr{{\mathbb R}}
\def\rn{{{\rr}^n}}

\def\nn{{\mathbb N}}

\def\fz{\infty}
\def\az{\alpha}

\def\dist{{\mathop\mathrm{\,dist\,}}}
\def\loc{{\mathop\mathrm{\,loc\,}}}
\def\lip{{\mathop\mathrm{\,Lip}}}

\def\lz{\lambda}
\def\dz{\delta}

\def\ez{\epsilon}

\def\gz{{\gamma}}

\def\boz{{\Omega}}

\def\wz{\widetilde}

\def\ls{\lesssim}
\def\gs{\gtrsim}

\def\bint{{\ifinner\rlap{\bf\kern.35em--}
\int\else\rlap{\bf\kern.45em--}\int\fi}\ignorespaces}

\def\bbint{{\ifinner\rlap{\bf\kern.35em--}
\hspace{0.078cm}\int\else\rlap{\bf\kern.45em--}\int\fi}\ignorespaces}

\def\esssup{{\rm \,esssup\,}}

\def\esup{\mathop\mathrm{\,esssup\,}}

\def\r{\right}
\def\lf{\left}

\newtheorem{thm}{Theorem}[section]
\newtheorem{lem}[thm]{Lemma}
\newtheorem{prop}[thm]{Proposition}
\newtheorem{rem}[thm]{Remark}
\newtheorem{cor}[thm]{Corollary}
\numberwithin{equation}{section}

\begin{document}

\arraycolsep=1pt

\title{\Large\bf
Intrinsic Geometry and Analysis of Diffusion Processes and $L^\fz$-Variational Problems
\footnotetext{\hspace{-0.35cm}
\noindent  {\it Key words and phases:}  Dirichlet form, Diffusion process, intrinsic distance,
 differential structure, $L^\fz$-variational problem, absolute minimizer
\endgraf Pekka Koskela was supported by the Academy of Finland
grant 131477. Nageswari Shanmugalingam was partially supported by grant \#200474 from
the Simons Foundation and by NSF grant \# DMS-1200915.
Yuan Zhou was supported by
Program for New Century Excellent Talents in University of Ministry of Education of China
and National Natural Science Foundation of China (Grant No. 11201015).
}}
\author{Pekka Koskela, Nageswari Shanmugalingam
and Yuan Zhou
}
\date{ }
\maketitle

\begin{center}
\begin{minipage}{13.5cm}\small
{\noindent{\bf Abstract}\quad
The aim of this paper is two-fold:

\quad
First, we obtain a better understanding of the intrinsic distance of
diffusion processes.  Precisely,
(i) for all $n\ge1$,
  the diffusion matrix $A$ is weak upper semicontinuous on $\boz$
if and only if the intrinsic differential and the local intrinsic distance
structures coincide;
(ii) if $n=1$, or if $n\ge2$ and $A$ is weak upper semicontinuous on $\boz$,
 the intrinsic distance and differential structures always coincide;
 (iii) if $n\ge2$ and $A$ fails to be weak upper semicontinuous on $\boz$,
the (non-) coincidence  of  the intrinsic distance and differential structures
depend on the geometry of the non-weak-upper-semicontinuity set of $A$.

\quad Second, for an arbitrary diffusion matrix $A$,
we show that the intrinsic distance completely  determines  the absolute
minimizer of the corresponding $L^\fz$-variational problem, and then
 obtain the  existence and uniqueness for given boundary data.
  We also give an example of a diffusion matrix $A$ for which there is an
absolute minimizer
that is not of class $C^1$.
When $A$ is  continuous, we also obtain the  linear approximation property  of
the absolute minimizer.
}
\end{minipage}
\end{center}

%
%
%
%
%
%
%

\tableofcontents
\contentsline{section}{\numberline{ } References}{36}

\section{Introduction\label{s1}}

Let $\boz\subset\rr^n$ be a domain (connected open subset).
Denote by $\mathscr A(\boz)$ the collection of all matrix-valued measurable
maps
$A=(a_{ij})_{1\le i,\,j\le n}:\boz\to\rr^{n\times n}$,
which are {\it elliptic}, that is,
for each $A\in\mathscr A(\boz)$, there exists a continuous function
$\lz:\boz\to[1,\,\fz)$ such that
\begin{equation}\label{e1.1}
\frac1{\lz(x)}|\xi|^2\le \langle A(x)\xi,\,\xi\rangle\le \lz(x)|\xi|^2
\end{equation}
for almost all $x\in\boz$ and all $\xi\in\rr^n$, where
$$\langle A(x)\xi,\,\xi\rangle=\sum_{i,\,j=1}^n\xi_ia_{ij}\xi_j.$$
An {\it Hamiltonian} associated to $A$  is given by
$H(x,\,\xi)=\langle A(x)\xi,\,\xi\rangle$.

Associated to each diffusion matrix $A\in\mathscr A(\boz)$,
there is a  ``Riemannian metric" (differential structure) on $\boz$:
for all $ x\in\boz$ and for each vector $\xi\in T_x\boz$,
the length of $\xi$ is given by $\sqrt{H(x,\,\xi)}$.
The corresponding differential operator $L_Au={\rm div}(A\nabla u)$ generates a
regular, strongly local bilinear Dirichlet energy form $\mathscr E_A $ with
domain
$\mathscr D(\mathscr E_A)$ in $L^2(\boz)$.
Notice that $C^\fz_c(\boz)$ is a core of $\mathscr E_A $,
$\mathscr D_\loc(\mathscr E_A)=W^{1,\,2}_\loc(\boz)$, and for all
$f,h\in \mathscr D(\mathscr E_A)$,
\[
\mathscr E_A(f,h)=\int_\boz\langle A(x)\nabla f(x),\nabla h(x)\rangle\, dx.
\]
For the details see for example \cite{fot}.
Moreover, the {\it intrinsic distance $d_A$} associated   to $A$ is defined by
$$d_A(x,\,y)=\sup\{u(x)-u(y)\}$$
for all $x,\,y\in \boz$,
where the supremum is taken over all $u\in  C(\boz)\cap W_\loc^{1,\,2}(\boz)$
such that
$H(x,\,\nabla u(x))\le 1$  almost everywhere. The ellipticity implies that
$d_A$ is locally comparable to the Euclidean distance.
We define the {\it pointwise Lipschitz constant} by setting
$$\lip_{d_A} u(x)=\limsup_{y\to x}\frac{|u(y)-u(x)|}{d_A(x,\,y)},$$
and its {\it local variant} $$|Du|_{d_A}(x)= \lim_{r\to0} \lip_{d_A}(u,\,B(x,\,r)),$$  where and in what follows
$$\lip_{d_A}(u,\,K)=\sup_{x,\,y\in K, y\ne x}\frac{|u(y)-u(x)|}{d_A(x,\,y)}.$$
Then $\lip_{d_A}(K)$ denotes the collection of all $u$ with
$\lip_{d_A}(u,\,K)<\fz$.
When $A=I_n$, $d_A$ is the Euclidean distance, and it  is always omitted in
the above notation.

Motivated by the work of Norris \cite{n97}, who showed
that the intrinsic distance  determines the small time asymptotics
of heat kernel, Sturm \cite{s97} asked the following question:
Is a diffusion process determined by the intrinsic distance?
In other words, do the  intrinsic differential and distance structures
coincide in the sense that for
$u\in \lip_{d_A}(\boz)$, $H(x,\,\nabla u(x))=(\lip_{d_A}u(x))^2$
almost everywhere?

The answer to this question is not always in the positive
as shown by Sturm's construction  \cite[Theorem 2]{s97}:
for each $A\in\mathscr A(\boz)$, there exists a $\wz A\in \mathscr A(\boz)$
with $d_{\wz A}=d_A$ but
$$\langle \wz A(x)\xi,\,\xi\rangle<\langle A(x)\xi,\,\xi\rangle$$
for all $\xi\in\rn\setminus\{0\}$;
see also \cite{kz} for a different example.
On the other hand, with the additional assumption that $A$ is continuous,
Sturm \cite[Proposition 4]{s97} proved that
the intrinsic differential and distance structures coincide.

The first aim of this paper is to obtain a better understanding on the
properties of $A$
that determine  the (non-)coincidence of intrinsic differential and distance
structures.
It turns out that weak upper semicontinuity  plays a critical role.
A function $u$ is said to be {\it weak upper semicontinuous}  at $x\in\boz$ if
there exists a set $E$ with $|E|=0$ such that
$$u(x)\ge\limsup_{(\boz\setminus E)\ni y\to x}u(y),$$
and is said to be weak upper semicontinuous  on $\boz$ if it is weak upper
semicontinuous at almost all
$x\in\boz$.
A diffusion matrix  $A$  is said to be \emph{weak upper semicontinuous} at
$x\in\boz$ (resp. on $\boz$)
if for every $\xi\in S^{n-1}$,
$\langle A(\cdot)\xi,\xi\rangle$ is weak upper semicontinuous at $x$
(resp. at almost all  $x\in\boz$).
Denote by ${\mathscr A}_{\rm wusc}(\boz)$
the collection of all $A\in\mathscr A(\boz)$
that are  weak upper semicontinuous on $\boz$.
We prove the following results.

\begin{enumerate}
\vspace{-0.3cm}
\item[(i)] For all $n\ge1$,
  the diffusion matrix $A$ belongs to $\mathscr A_{\rm wusc}(\boz)$
if and only if the intrinsic differential and the local intrinsic distance
structures coincide in the sense that
for all $u\in C^1_\loc(\boz)$, $|Du|^2_{d_A}(x)= H(x,\,\nabla u(x)) $ almost
everywhere; see Theorem \ref{t2.3}.

\vspace{-0.3cm}

\item[(ii)] If $n=1$, or if $n\ge2$ and $A\in\mathscr A_{\rm wusc}(\boz)$, then
 the intrinsic distance and differential structures always coincide, that is,
 for all $u\in\lip(\boz)$,
 $(\lip_{d_A}u(x))^2= H(x,\,\nabla u(x)) $ almost everywhere; see Theorems
\ref{t2.1} and \ref{t2.2}.

\vspace{-0.3cm}
\item[(iii)] If $n\ge2$ and $A\notin\mathscr A_{\rm wusc}(\boz)$,
the (non-) coincidence  of  the intrinsic distance and differential structures
depend on the geometry of the non-weak-upper-semicontinuity set of $A$.
Indeed, we construct two examples via a large Cantor set and a large
Sierpinski carpet to show that both coincidence and noncoincidence  may happen;
see, respectively,  Theorem \ref{t3.1} and Proposition \ref{p3.1}.
\end{enumerate}

The proofs of Theorems \ref{t2.2} and \ref{t2.3} rely on the (key)
Lemmas \ref{l2.1} and
 \ref{l2.2}.
The proof of Theorem \ref{t3.1} is more intricate;
we use an approximation of the distance by Norris \cite{n97} to derive
some careful estimates on a good set of the distance function based
on geometric properties of our Sierpinski carpet.
Proposition \ref{p3.1} uses the geometry of the complement of our
large Cantor set.

We also consider the $L^\fz$-variational problem associated with an
arbitrary matrix-valued map $A\in\mathscr A(\boz)$:
the goal is to study the local minimizers of the functional
$$F(u;U)=\esssup_{x\in U} H(x,\,\nabla u(x))$$
over the class of Lipschitz functions on $U\Subset\boz$ with a given boundary
data.
This study was initiated by Aronsson \cite{a1,a2,a3,a4} in the case
$H(x,\,\xi)=|\xi|^2$, that is, $A=I_n$.
He introduced the idea of  absolute minimizer, that is, minimize  $F$ on all
open subset of $U$.
To be precise, let $U$ be an open subset such that $\overline U\subset \boz$.
A function $u\in \lip(U)$ is said to be an
{\it absolute  minimizer  for $H$} on $U$ if for every open subset
$V\Subset U$  and $v\in \lip(V)\cap C(\overline V)$ with
  $u|_{\partial V}= v|_{\partial V}$, we have
$$\esssup_{x\in V} H(x, \nabla u(x)) \le {\esssup}_{x\in V} H(x, \nabla v(x)).$$
Moreover, given a function $f\in \lip (\partial U)$,  $u\in \lip(U)$ is said
to be an
 {\it absolutely minimizing Lipschitz extension } of $f$ if $u$ is an
absolute  minimizer  for $H$
and $u|_{\partial U}=f$.
In recent years, the study of the $L^\fz$-variational problem, even
for more general Hamiltonians but with some smoothness,
has advanced significantly; see \cite{acj} for a survey and \cite{ceg,j93} for
some seminal works.
The $L^\fz$-variational problem is still interesting even if the Hamiltonian
is not smooth or even continuous.
See for example \cite{acj,gpp,cp,cpp} and the reference therein.
In this case, one cannot always derive an Aronsson equation from the
$L^\fz$-variational problem.

Our results concerning absolute minimizer are as follows:
\begin{enumerate}
\vspace{-0.3cm}
\item[(iv)]
  For arbitrary $A\in\mathscr A(\boz)$ and the Hamiltonian
$H(x,\,\xi)=\langle A(x)\xi,\,\xi\rangle$,
we show that the absolute minimizer is completely determined by the intrinsic
distance, and then obtain the existence and uniqueness of the absolute
minimizer  given a boundary data;
see  Theorem \ref{t4.1}.
Consequently, if $A,\,\wz A\in\mathscr A(\boz)$ and $d_{A}=d_{\wz A}$,
then  given the boundary data, the absolute minimizers associated to $A$ and
$\wz A$ coincide.
\vspace{-0.3cm}
\item[(v)]
Associated to the diffusion matrix $A\notin\mathscr A_{\rm wusc} (\rn)$
given in Subsection 3.2, we show in Proposition \ref{c4.3} that there is an
absolute minimizer $u$ on $(0,\,1)^n$ which fails to be $C^1$.
 This example indicates that perhaps weak upper semicontinuity of $A$ is
needed in order for the corresponding absolute minimizers to be of class $C^1$.

\vspace{-0.3cm}
\item[(vi)]
{We obtain  in Theorem \ref{t5.1} the linear approximation property of the
absolute minimizer
at all points of continuity of $A$, and hence  at  all points when $A$ is
continuous on $\boz$.}
\end{enumerate}

The proof of Theorem \ref{t4.1} relies on the (crurial) Lemma \ref{l2.1} and
Lemma \ref{l4.6},
which allows us to describe the absolute minimzer via the pointwise Lipschitz
constant.
Then the existence of the absolute minimizer follows from \cite{jn},
while the uniqueness will be proved following the idea of \cite{as}
(see \cite{pssw} for an earlier proof via the tug of war).
 Proposition \ref{c4.3} follows from Theorem \ref{t3.1} and
properties of absolute minimizers.
The proof of Theorem \ref{t5.1} borrows the blow-up ideas of \cite{ceg}, but
due to the change of distance in the blow-up process,
a detailed study is necessary.

The $C^1$-regularity of the absolute minimizer is still open except for
the case $n=2$ and $A=I_n$.
Precisely, if $A=I_n$, Savin \cite{s05} obtained the $C^1$-regularity  of the
absolute minimizer when $n=2$  (see also
\cite{wy} for a homogeneous norm and \cite{es})
while Evans-Smart \cite{es12} obtained the everywhere differentiability
when $n\ge 3$.
All the proofs in \cite{s05,es,es12,wy} rely on the linear approximation
property;
indeed, controlling the convergence of different  sequences appearing in
the linear approximation.
For an arbitrary continuous or even $C^1$-continuous  $A$,
we do not know if it is possible to obtain the
everywhere differentiability
by controlling the linear approximation process provided in Theorem \ref{t5.1}
as done in \cite{s05,es,es12,wy}.

Finally, we state some {\it conventions}. Throughout the paper,
we denote by $C$ a {\it positive
constant} which is independent
of the main parameters, but which may vary from line to line.
Constants with subscripts, such as $C_0$, do not change
in different occurrences. The {\it notation} $A\ls B$ or $B\gs A$
means that $A\le CB$. If $A\ls B$ and $B\ls A$, we then
write $A\sim B$.
Denote by  $\nn$ the {\it set of positive integers}.
If $V$ is   a bounded  open set with $\overline V\subset U$, we simply write $V\Subset U$.
We use $C(\boz)$ to denote the continuous function on $\boz$
while $C^1(\boz)$  the  function with continuous gradient on $\boz$.
For any locally integrable function $f$,
we denote by $\bbint_E f\,d\mu$ the {\it average
of $f$ on $E$}, namely, $\bbint_E f\,d\mu\equiv\frac 1{\mu(E)}\int_E f\,d\mu$.

\section{Case $n=1$ or $A\in\mathscr A_{\rm wusc}(\boz)$:  $H(\cdot,\,\nabla u )=(\lip_{d_A}u )^2$}\label{s2}

We first show that if $n=1$, or if $n\ge2$ and $A\in\mathscr A_{\rm wusc}(\boz)$,
then the intrinsic distance and differential structures always coincide in
the sense that
for all   $u\in\lip(\boz)$, $(\lip _{d_A}u(x))^2=H(x,\,\nabla u(x))$ almost
everywhere; see Theorems \ref{t2.1} and \ref{t2.2}.
Then, for all $n\ge1$, we prove that $A\in\mathscr A_{\rm wusc}(\boz)$
if and only if the intrinsic differential and the local intrinsic distance
structures coincide in the sense that
for all   $u\in C^1_\loc(\boz)$, $|Du|^2_{d_A}(x)=H(x,\,\nabla u(x))$ almost
everywhere; see Theorem \ref{t2.3}.

\begin{thm}\label{t2.1}
If $n=1$, then for all $u\in\lip_{d_A}(\boz)$,
$\lip_{d_A}u=\sqrt A|u'|$ almost everywhere.
\end{thm}

\begin{proof}
By the continuity of $\lambda$ associated with the ellipticity condition of $A$,
we have $\text{Lip}_{d_A}(U)=\text{Lip}(U)$ for $U\Subset\Omega$.
To prove  that $\lip_{d_A}u(x)\le\sqrt{A(x)}|u'(x)|$
for almost all $x\in\Omega$, notice that
\begin{eqnarray*}
\lip_{d_A}u(x)&&=\limsup_{y\to x}\frac{|u(y)-u(x)|}{d_A(x,\,y)}\\
&&\le  \limsup_{y\to x}\frac{|u(y)-u(x)|}{|x-y|}\limsup_{y\to x}\frac{|x-y|} {d_A(x,\,y)}\\
&&=|u'(x)|\limsup_{y\to x}\frac{|x-y|} {d_A(x,\,y)}.
\end{eqnarray*}
Here we used the Rademacher theorem, according to which locally Lipschitz
continuous
functions on $\mathbb{R}$ are
differentiable almost everywhere.
Thus it suffices to check that for almost all $x\in\boz$,
\begin{equation}\label{e3.x2}
\limsup_{y\to x}\frac{|x-y|} {d_A(x,\,y)}\le \sqrt{A(x)}.
\end{equation}
This is reduced to showing that, for any $\ez>0$,
there exists a constant $\dz>0$ and a Lipschitz
continuous function $w$ such that $A(x)|w'(x)|^2\le1$
and for all $y\in(x-\dz,\,x+\dz)$,
\begin{equation}\label{e3.x3}
|y-x|\le (1+\ez)\sqrt{A(x)}|w(y)-w(x)|.
\end{equation}
Indeed, from this and the definition of $d_A$, we know that
\[
\sqrt{A(x)}\,d_A(x,\,y)\ge \sqrt{A(x)}\,|w(y)-w(x)|\ge \frac{1}{1+\ez}|y-x|,
\]
which implies  \eqref{e3.x2} by the arbitrariness  of $\ez$. Towards
\eqref{e3.x3}, take $$w(z)= \int_{x}^z\frac{1}{\sqrt{A(s)}}\,ds$$
for $z\in\boz$.
Notice that the lower bound of $A$ guarantees that
$\frac{1}{\sqrt{A}}\in L^1_\loc(\boz)$, and so
$w^\prime(z)=A(z)^{-1/2}$ for almost all $z\in\mathbb{R}$.
Let $I_{x,\,y}=[x,\,y]$ if $x<y$ ($[y,\,x]$ if $x>y$).
By Lebesgue's differentiation theorem, for almost all $x\in\boz$, we can
find $\dz>0$ such that
  whenever $y\in(x-\dz,\,x+\dz)$,
 $$ \frac{|w(y)-w(x)|}{|x-y|}=\bint_{I_{x,\,y}}\frac{1}{\sqrt{A(s)}}\,ds\ge\frac{1}{(1+\ez)\sqrt{A(x)}},$$
which implies \eqref{e3.x3}.

On the other hand, to prove $\sqrt{A(x)}|u'(x)|\le \lip_{d_A}u(x)$ for almost
all $x\in\boz$,
observe that at points $x$ of differentiability of $u$ (by the classical
Rademacher's theorem,
almost every $x$ is such a point),
\begin{eqnarray*}
|u'(x)| &&=\lim_{y\to x}\frac{|u(y)-u(x)|}{|x-y|}\\
&&\le  \limsup_{y\to x}\frac{|u(y)-u(x)|}{d_A(x,\,y)}\limsup_{y\to x}\frac{d_A(x,\,y)}{|x-y|}\\
&&=\lip_{d_A}u(x)\limsup_{y\to x}\frac{d_A(x,\,y)}{|x-y|}.
\end{eqnarray*}
By the definition of $d_A$, for any $\ez>0$ and any fixed $y$, there exists a
function $v$ such that $A(z)|v'(z)|^2\le1$ for almost all $z\in\boz$ and
$d_A(x,\,y)\le(1+\ez)|v(x)-v(y)|$,
which implies that
$$ \frac{d_A(x,\,y)}{|x-y|}\le(1+\ez) \frac{|v(x)-v(y)|}{|x-y|}\le
\frac{(1+\ez)}{|x-y|} \int_{I_{x,\,y}}|v'(s)|\,ds \le \frac{(1+\ez)}{|x-y|} \int_{I_{x,\,y}}\frac1{\sqrt{A(s)}}\,ds
.$$
If $x$ is a Lebesgue point of $\frac1{\sqrt{A}}$,
there exists a $\dz>0$ such that whenever $y\in(x-\dz,\,x+\dz)$,
$$\frac{d_A(x,\,y)}{|x-y|}\le (1+\ez) ^2\frac1{\sqrt{A(x)}},$$
which is as desired.
\end{proof}

 \begin{thm}\label{t2.2}
If $n\ge2$ and   $A\in\mathscr A_{\rm wusc}(\boz)$,
then  the intrinsic distance and differential structures coincide. That is,
given Lipschitz function $u$ on $\boz$ {\rm(}with respect to the
Euclidean metric\rm{)},
for almost every $x\in\boz$
we have
\[
    (\text{Lip}_{d_A} u(x))^2=\langle A(x)\nabla u(x),\nabla u(x)\rangle.
\]
 \end{thm}

 To prove Theorem \ref{t2.2}, we first notice that
 the distance $d_A$ is locally comparable to the Euclidean distance,
 and hence $ (\boz,\,d_A,\,dx)$ satisfies the local doubling property in the
sense that if $U$ is open and $U\Subset \boz$, there exists a constant
depending on $U$ and $A$
 such that for each  $x\in U$  and  $0<r<\min\{\text{diam}(U), d_A(x,\,\boz^\complement )\}/4$,
  $$\mu(B_{d_A}(x,\,2r))\le C(A,\,U)\mu(B_{d_A}(x,\, r)).$$
Here and in what follows, $d_A(x,\,K )=\inf_{x\in K}d_A(x,\,z)$
and if $d_A$ is the Euclidean distance, we use the notation $d_{\rn}(x,\,K )$.

Therefore, applying \cite[Theorem 2.1]{kz} and its remark, we conclude with the
following.

  \begin{lem}\label{l2.1}
For each $u\in\lip_{d_A}(\boz)$,
$H(x,\,\nabla u(x))\le (\lip_{d_A}u(x))^2$ for almost all $x\in\boz$.
  \end{lem}

To obtain the reverse relation,
 we need the following result.

\begin{lem}\label{l2.2}
Assume that $n\ge2$.
Let $x_0\in\boz$ and $0<r< d_\rn(x_0,\, \boz^\complement)$.
If the diffusion matrix $\wz A$
is a constant positive definite symmetric matrix $A$ on the Euclidean ball
${B(x_0,\,r)}$,
then, for the function $u(y)=|A^{1/2}\xi|^{-1}\langle \xi, y\rangle$ with $\xi\in S^{n-1}$,
we have  $ \lip_{d_{\wz A}}u \le 1$ on $B(x_0,\,r)$.
\end{lem}

\begin{proof}
It suffices to show that, for every $x\in B(x_0,\,r)$,
there exists $\dz\in(0,\,r-|x-x_0|)$
such that for each fixed $y\in B(x,\,\dz)$, we can find
a  function $\wz u_{x,\,y}$  on $\boz$ satisfying

(i) $\wz u_{x,\,y}(z)=u(z)$  for all $z$ in the line segment joining $x$ and $y$,

(ii) $ \langle \wz A(z)\nabla\wz u_{x,\,y}(z),\,\nabla\wz u_{x,\,y}(z)\rangle\le1 $ for almost all $z\in\boz$.

\noindent Indeed, from the definition of the intrinsic distance, the existence
of such a function will lead to
$|\wz u_{x,\,y}(z)-\wz u_{x,\,y}(w)|\le d_{\wz A}(z,\,w)$ for all $z,\,w\in\boz$,
which gives the desired inequality
$|u(x)-  u (y)|\le d_{\wz A}(x,\,y)$ when choosing $z,\,w$ as  $x,\,y$. Hence we have $\lip_{d_A}u(x)\le 1$ for all $x\in B(x_0,\,r)$.

To this end, we first consider  the case $A=\lz I_n$. Set
 \begin{equation}\label{e3.1}
 \wz u_{x,\,y}(z)=\lf\{\begin{array}{ll}
 [u(y)-u(x)]\frac{|z-x|}{|y-x|}+u(x)&\quad if \quad |z-x|\le |y-x|\\
u(y)&\quad if \quad |z-x|\ge |y-x|.
 \end{array}\r.
 \end{equation}
Obviously, $\wz u_{x,\,y}$ satisfies (i). To see (ii),
for $x\ne z\in B(x,\,|x-y|)\subset B(x,\,\dz)\subset B(x_0,\,r)$, we have
$$|A^{1/2}\nabla \wz u_{x,\,y}(z)|=  \frac{|u(y)-u(x)|}{|y-x|}\frac{|A^{1/2}(z-x)|}{|z-x|}
= \frac{|\langle \xi, y-x\rangle|}{ |A^{1/2}\xi|} \frac{|A^{1/2}(z-x)|}{|y-x||z-x|}\le1;
$$
while when  $z\in \boz\setminus \overline{B}(x,\,|x-y|)$ we have
$|\wz A^{1/2}\nabla \wz u_{x,\,y}(z)|= 0\le 1$.

For a more general constant positive definite symmetric matrix  $A$, we
modify the above construction as follows,
following the idea given above.
Notice that  there exist  $\dz_1,\dz\in(0,\,r)$ such that
$A^{1/2} B(0,\,\dz_1)\subset B(0,\,r )$ and $A^{-1/2}B(0,\,\dz)\subset B(0,\,\dz_1)$.
Thus, for every $y\in B(x,\,\dz)$, we have  $|A^{-1/2}(y-x)|<\dz_1$  and  hence
$$\{z\in \rn:\ |A^{-1/2}(z-x)|\le |A^{-1/2}(y-x)|\}\subset A^{1/2}B(0,\,\dz_1)+\{x\}\subset B(x_0,\,r).$$
For a given pair $x,\,y$, set
 \begin{equation}\label{e3.2}
 \wz u_{x,\,y}(z)=\lf\{\begin{array}{ll}
 [u(y)-u(x)]\frac{|A^{-1/2}(z-x)|}{|A^{-1/2}(y-x)|}+u(x)&\quad if \quad |A^{-1/2}(z-x)|\le |A^{-1/2}(y-x)|,\\
u(y)&\quad if \quad |A^{-1/2}(z-x)|\ge |A^{-1/2}(y-z)|.
 \end{array}\r.
 \end{equation}
 We still need to check that if  $|A^{-1/2}(z-x)|\le |A^{-1/2}(y-x)|$, then
 $|A^{1/2}\nabla \wz u_{x,y}(z)|\le1$.
 Indeed, notice that for $y\in\mathbb{R}^n$,
 \[
   \langle A^{1/2}\xi, A^{-1/2}y\rangle = \langle A^{1/2}\xi,A^{1/2}A^{-1}y\rangle
       =\langle\xi, A\, A^{-1}y\rangle=\langle\xi,y\rangle,
 \]
 and so
 $\langle \xi, y\rangle=\langle A^{1/2}\xi, A^{-1/2}y\rangle$. Furthermore, for $z\in S^{n-1}$,
 $$\nabla |A^{-1/2}z|=\frac{A^{-1}z}{|A^{-1/2}z|},$$
 and so by the Cauchy-Schwarz inequality, we have
 $$|A^{1/2}\nabla \wz u_{x,y}(z)|=
\frac{|\langle \xi, y-x\rangle|}{ |A^{1/2}\xi|} \frac{|A^{1/2}A^{-1}(z-x)|}{|A^{-1/2}(y-x)||A^{-1/2}(z-x)|}
\le1,$$
as desired.
\end{proof}

\begin{lem}\label{comp-dist}
Given $A\in\mathscr{A}(\Omega)$ and $x\in\Omega$, we have
\[
   \liminf_{x\ne y\to x}\frac{d_A(y,x)}{|y-x|}\ge \frac{1}{\sqrt{\lambda(x)}}.
\]
Furthermore,
\[
  \limsup  _{x\ne y\to x}\frac{d_A(y,x)}{|y-x|}\le \sqrt{\lambda(x)}.
\]
\end{lem}

\begin{proof}
   We fix $x\in\Omega$ and $r>0$ such that $\overline{B}(x,r)\subset\Omega$. Set
\[
   \lambda_x(r)=\sup_{z\in B(x,r)}\lambda(z).
\]
Notice that, by the continuity of $\lambda,$ we have $\lim_{r\to 0}\lambda_x(r)=\lambda(x)$.
For the function $u_x$ given by
\[
   u_x(z)=\frac{1}{\sqrt{\lambda_x(r)}}(r-|z-x|)_+,
\]
we have that $\nabla u_x=0$ on $\Omega\setminus\overline{B}(x,r)$ and so $H(z,\nabla u_x(z))=0\le 1$
when $z\in\Omega\setminus\overline{B}(x,r)$. When $x\ne z\in B(x,r)$, we have
$\nabla u_x(z)=\lambda_x(r)^{-1/2} |z-x|^{-1}(z-x)$. Therefore, by the
ellipticity condition of $A$,
$H(z,\nabla u_x(z))=\langle A(z)\nabla u_x(z),\nabla u_x(z)\rangle\le 1$.
It follows that, by
the definition of $d_A$, when $x\ne y\in B(x,r)$,
\[
  d_A(x,y)\ge |u_x(y)-u_x(x)|=\frac{1}{\sqrt{\lambda_x(r)}}\, |y-x|,
\]
from which the first part of the claim follows.

For the second part, notice that if $x,y\in\Omega$, then there is a function
$w$ on $\Omega$
with $d_A(x,y)\le |w(y)-w(x)|+\epsilon$ and $H(z,\nabla w(z))\le 1$ for
almost every $z\in\Omega$.
Let $E$ be the set of points at which this inequality fails.
Then the Lebesgue measure of $E$ is zero.
By the ellipticity property of $A,$ it follows that for almost every
$z\in\Omega\setminus E$,
$|\nabla w(z)|^2\le \lambda(z)$. By an argument using Fubini's theorem, for
each $\eta>0$
there is a point $y_\eta\in B(y,\eta)$ and a point $x_\eta\in B(x,\eta)$
such that the intersection of the Euclidean line segment $[x_\eta,y_\eta]$
connecting $x_\eta$ to
$y_\eta$ with $E$ has $1$-dimensional Hausdorff measure zero. Thus
\[
|w(y_\eta)-w(x_\eta)|\le\int_{[x_\eta,y_\eta]}|\nabla w|\, ds \le \left[\sup_{p\in B(x,2{d_A}(x,y))}\sqrt{\lambda(p)}\right]
             |x_\eta-y_\eta|.
\]
Letting $\eta\to 0$ and using the fact that $w$ is continuous, we obtain
\[
   d_A(x,y)\le  \left[\sup_{p\in B(x,2{d_A}(x,y))}\sqrt{\lambda(p)}\right]\, |x-y|\ + \epsilon.
\]
Letting $\epsilon\to 0$ we obtain
\[
   d_A(x,y)\le \left[\sup_{p\in B(x,2{d_A}(x,y))}\sqrt{\lambda(p)}\right]\, |x-y|,
\]
from which the second part of the claim follows.
\end{proof}

\begin{proof}[Proof of Theorem \ref{t2.2}.]
Let $u\in\lip_{d_A}(\boz).$ Then, $u$ is also locally Lipschitz continuous
with respect to the Euclidean metric on $\boz$.
Notice that by Lemma \ref{l2.1},   we always have
$H(x,\,\nabla u(x))\le (\lip_{d_A}u(x))^2$ for almost all $x\in\boz$.
Now  we will show that $(\lip_{d_A}u(x))^2\le H(x,\,\nabla u(x))$
for every $x\in\boz$ at which $A$ is weak upper semicontinuous and  $u$ is
differentiable.
Fix such an $x\in\boz$.
If $\nabla u(x)= 0$, then $|u(x)-u(y)|=o(|x-y|)$,
which implies by Lemma~\ref{comp-dist} that
\[
   \lip_{d_A} u(x)\le \lip\, u(x)\, \sqrt{\lambda(x)}=0=H(x,\nabla u(x)).
\]
If  $\nabla u(x)\ne0$, take $\xi=\frac{\nabla u(x)}{|A^{1/2}(x)\nabla u(x)|}$.
Then, by Lemma~\ref{comp-dist} again, together with the fact that $u$ is
differentiable at $x$,
\begin{eqnarray}\label{e2.5}
\lip_{d_A}u(x)&&=\limsup_{y\to x}\frac{|u(y)-u(x)|}{d_A(x,y)}
\\
&&\le\limsup_{y\to x}\frac{|u(y)-u(x)-\langle \nabla u(x),\,y-x\rangle|}{d_A(x,y)}+
\limsup_{y\to x}\frac{|\langle \nabla u(x) ,\,y-x\rangle|}{d_A(x,y)}\nonumber\\
\le&& \sqrt{\lambda(x)}\, \lim_{y\to x}\frac{|u(y)-u(x)-\langle \nabla u(x),\,y-x\rangle|}{|x-y|}+
\limsup_{y\to x}\frac{|\langle \nabla u(x) ,\,y-x\rangle|}{d_A(x,y)}\nonumber\\
&&\le 0+| A^{1/2}(x)\nabla u(x)|\limsup_{y\to x}\frac{|\langle \xi,\,y-x\rangle|}{d_A(x,y)}.\nonumber
\end{eqnarray}
Observe that $$| A^{1/2}(x)\nabla u(x)|^2=\langle A^{1/2}(x)\nabla u(x),\,A^{1/2}(x)\nabla u(x) \rangle
=H(x,\, \nabla u(x) ).$$
Let $w(y)=\langle \xi,\,y\rangle$. It suffices to prove that
\begin{equation}\label{e3.x1}
 \limsup_{y\to x}\frac{|w(y)-w(x)|}{d_A(x,y)}\le 1.
 \end{equation}
To this end, notice that $\nabla w(y)=\xi$, and hence
$H(y,\,\nabla w(y))=H(y,\, \xi )$ for all $y\in\boz$.
By the weak upper semicontinuity of $A$ at $x$, there exists $\dz\in(0,\,r)$
such that
for all $y\in B(x,\,\dz)$,
$$ H(y,\, \xi ) \le (1+\ez)H(x,\, \xi )=(1+\ez).$$
Set $\wz A(z)=(1+\ez)A(x)$ for $z\in B(x,\,\dz)$ and $\wz A(z)=A(z)$ for $z \notin B(x,\,\dz)$.
It can be directly seen that $d_{\wz A}\le d_A$.
We consider the function $v(y)=\frac{1}{\sqrt{1+\ez}}w(y)$ and $\eta=|\xi|^{-1}\xi\in S^{n-1}$; to this choice
we apply Lemma \ref{l2.2} to obtain
$$\limsup_{y\to z}\frac{|w(y)-w(z)|}{d_A(z,y)}\le  \limsup_{y\to z}\frac{|w(y)-w(z)|}{d_{\wz A}(z,y)}\le \sqrt{1+\ez}.$$
The arbitrariness  of $\ez>0$ leads to \eqref{e3.x1}.
\end{proof}

\begin{thm}\label{t2.3}
A diffusion matrix $A$ belongs to $\mathscr A_{\rm wusc}(\boz)$  if and only if
$|Du|^2_{d_A} =H(\cdot,\,\nabla u )$ almost everywhere
for all $u\in C^1(\boz)$.
\end{thm}

\begin{proof}
Suppose that whenever $u\in C^1(\boz)$ we have
$|Du|_{d_A}^2(x)=H(x,\,\nabla u(x))$ almost everywhere.
For each $\xi\in S^{n-1}$,
taking $u(x)=\langle \xi,\,x\rangle$, we know that
$H(x,\,\xi)=|Du|_{d_A}^2(x)$ almost everywhere,
 and hence is weak upper semicontinuous on $\boz$ because $x\mapsto |Du|_{d_A}(x)$
 is upper semicontinuous.

Suppose now that $A\in\mathscr A_{\rm wusc}(\boz)$. Then by Theorem~\ref{t2.1} and Theorem~\ref{t2.2},
for all $u\in \lip_{d_A}(\boz)$ and
almost all $x\in\boz$ we have
$$H(x,\,\nabla u(x))\le (\lip_{d_A}u(x))^2\le |Du|^2_{d_A}(x).$$
To see that $ |Du|^2_{d_A}(x)\le H(x,\,\nabla u(x))$ almost everywhere for $u\in C^1(\boz)$, we give a 3-step argument.

{\it Step 1. }
Let $\xi\in S^{n-1}$ and consider the function $u(y)=\langle\xi,\,y\rangle$.
Then $\nabla u(x)=\xi$.
To prove that
$|Du|^2_{d_A}(x)\le  H(x,\,\xi)$, it suffices to check that  for almost every $x\in\boz$,
\begin{equation}\label{e2.x2}
|Du|^2_{d_A}(x)\le (1+\ez) H(x,\,\xi)
\end{equation} for any $\ez>0$.
The following argument is similar to that of Theorem \ref{t2.2}. Let $x$ be a point of weak upper
semicontinuity of $A$.
For each fixed $\ez>0$, we know that there exists $r>0$ such that for almost all
$y\in B(x,\,r)$,  $H(y,\,\xi)\le (1+\ez)H(x,\,\xi)$.
Let $$\wz A(y)=(1+\ez)A(x)1_{B(x,\,r)}(y) I_n+A(y)1_{\boz\setminus \overline {B(x,\,r)}}.$$
The corresponding intrinsic distance $d_{\wz A}$ is no more than $d_A$, and hence
$|Du|_{d_A}\le |Du|_{d_{\wz A}}$ everywhere.  By Lemma \ref{l2.2}, for any $y,\,z\in B(x,\,r/4)$, we have
$$|u(y)-u(z)|\le d_{\wz A}(z,\,y)$$ and hence,  $|Du|^2_{d_{\wz A}}(x)\le  {\langle \wz A(x)\xi,\,\xi\rangle}$,
which implies  \eqref{e2.x2} as desired.

{\it Step 2. } If $u\in C^1(\boz)$  and $\nabla u(x)=0$,
then for any $\ez>0$, by the continuity of $\nabla u$,
there exists a ball $B(x,\,r)$
such that $\|\nabla u\|_{L^\fz(B(x,\,r))}\le \ez$.
With the aid of Lemma~\ref{comp-dist}, we obtain
$$|u(y)-u(z)|\le \ez|y-z|\ls \ez d_A(y,\,z)$$
which implies that $|Du|_{d_A}(x)\ls \ez$, and
hence $|Du|_{d_A}(x)=0$ due to the arbitrariness  of $\ez>0$.

{\it Step 3. } If $u\in C_\loc^1(\boz)$ with $\nabla u(x)\ne 0$, then
let $$\wz u(y)=u(y)-\langle\nabla u(x),\,y\rangle$$ for $y\in\boz$.
Then $\wz u$ is of class $C^1(\boz)$ with $\nabla\wz u(x)=0$, which implies that $|D\wz u|_{d_A}(x)=0$ by Step 2.
Moreover, since $\langle A(y)\nabla u(x),\nabla u(x)\rangle$ is weak upper semicontinuous at $y=x$,
by Step 1 we have
$$|D  u|_{d_A}(x)\le  |D\wz u|_{d_A}(x)+|D(\langle \nabla u(x),\cdot\rangle)|_{d_A}(x)
\le\sqrt{H(x,\,\nabla u(x))} $$
as desired.
\end{proof}

\begin{rem}\label{r2.1} \rm
(i) We cannot replace the function class $C^1(\boz)$ by
$\lip_\loc(\boz)$ in the above Theorem \ref{t2.3}.
Indeed, there exists a function
$u\in\lip(\boz)$ such that $|\nabla u|^2$ is not weak upper semicontinuous on $\boz$, hence the above theorem fails
for $A=I_n$.
For example $$u(x_1,\,x_2)=\int_0^{x_1} 1_{C_{\bf a}}(z_1)\,dz_1,$$
where $C_{{\bf a}}\subset\mathbb{R}$ is a Cantor set of positive $1$-dimensional Lebesgue measure, and
$1_{C_{\bf a}}:\mathbb{R}\to\mathbb{R}$ is the characteristic function of $C_{\bf a}$.
For a construction of such a Cantor set $C_{\bf a}$ see Section~3.

 (ii) Generally, for every open set $U$ with $\overline U \subset \boz$ and $u\in \lip(\boz)$, we have
$|Du|_{d_A}\le C_U\lip_{d_A}u$ almost everywhere on $U$,
 where $C_U\ge 1$ is a constant. The proof of this is not trivial.
\end{rem}

Finally, we point out a relation between weak upper  semicontinuity
and the Eikonal equation. The Eikonal equation is necessary to obtain the
coincidence of
the intrinsic differential and distance structures. The Eikonal equation states
that
$$\langle A(\cdot)\nabla d_{A;\,x_0}(\cdot),\, \nabla d_{A;\,x_0} (\cdot)\rangle=1$$  almost everywhere for each $x_0\in\boz$,
where  $d_{A;\,x_0}(x)=d_A(x,\,x_0)$. When $A=I_n$ the statement of the
Eikonal equation is that whenever $x_0\in\mathbb{R}^n$, the function
${d_A}(\cdot,x_0)$
satisfies $|\nabla {d_A}(\cdot,x_0)|=1$ almost everywhere in $\mathbb{R}^n$ (indeed, everywhere in
$\mathbb{R}^n\setminus\{x_0\}$).

\begin{prop}\label{p2.1}
For each $x_0\in\boz$,
$H(x,\, \nabla d_{A;\,x_0}(x))=1$ almost everywhere
 if and only if $H(\cdot,\, \nabla d_{A;\,x_0} )$ is weak upper semicontinuous on $\boz$.
\end{prop}

\begin{proof}
If $H(x,\, \nabla d_{A;\,x_0}(x))=1$ almost everywhere,
then obviously it is weak upper semicontinuous.
Conversely, assume that $H(\cdot,\, \nabla d_{A;\,x_0} )$ is weak upper semicontinuous on $\boz$.
It suffices to show that for each point $x\in \boz$ and all sufficient small $r>0$,
\begin{equation}\label{e2.x1}
 \lf\| H(\cdot,\, \nabla d_{A;\,x_0} )\r\|_{L^\fz(B_{d_A}(x,\,r))}=1.
\end{equation}
Indeed, if this is true, then for almost all $x$, the weak upper semicontinuity
leads to
$$1\ge H(x,\, \nabla d_{A;\,x_0}(x) )\ge
\limsup_{r\to 0}\lf\|H(\cdot,\, \nabla d_{A;\,x_0})\r\|_{L^\fz(B_{d_A}(x,\,r))}=1.$$

We prove \eqref{e2.x1} by contradiction.
Assume that \eqref{e2.x1}  fails for some $x_0\in\boz$ and some decreasing
sequence  $\{r_k\}$ which converges to $0$ as $k\to\fz$.
By Lemma~\ref{comp-dist}  and its proof, $d_A$ is comparable to the Euclidean distance.
Hence for sufficiently large $k$ we have $B_{d_A}(x_0,\, r_k)\subset \boz$.
Moreover, since we already know from Lemma~\ref{l2.1} applied to the function ${d_A}(\cdot,x_0)$ that
$ \lf\|H(\cdot,\, \nabla d_{A;\,x_0}  )  \r\|_{L^\fz(B_{d_A}(x,\,r_k))}\le1$,
 by our assumption there must be a positive number $\ez_k<1$ such that
  $$\lf\| H(\cdot,\, \nabla d_{A;\,x_0}  )\r\|_{L^\fz(B_{d_A}(x,\,r_k))}\le 1-\ez_k.$$
Taking $$u_k(x)=\frac 1{\sqrt{1-\ez_k}}\min\{{d_A}(x_0,\,x),\,r_k\},$$  we have $u_k\in \lip(\boz)$
with
 $\lf\| H(\cdot,\,\nabla u_k)\r\|_{L^\fz( B_{d_A}(x_0,\,r_k) )}\le1 $  and
  $ H(z,\,\nabla u_k(z))=0 $ for $ z\in \boz\setminus B_{d_A}(x_0,\,r_k)$.
Hence  $u_k$ satisfies the conditions in the definition of $d_A(x,x_0)$, and so
$$d_A(x_0,\,x)\ge u_k(x)-u_k(x_0)=\frac 1{\sqrt{1-\ez_k}}d_A(x_0,\,x),$$
which is a contradiction.
\end{proof}

\begin{rem}\label{r2.2}\rm As shown in \cite[Section 7]{kz},
the Eikonal equation determines the asymptotic behavior of the gradient of
heat kernel for a regular, strongly local Dirichlet form
on a compact underlying space.
We do not know if it is possible to deduce the coincidence of the
intrinsic differential and distance structures from the Eikonal equation.
\end{rem}

\section{Case $n\ge2$ and $A\notin\mathscr A_{\rm wusc}(\boz)$: $H(\cdot,\,\nabla u )=(\ne)(\lip_{d_A}u )^2$ }\label{s3}

In this section, we always assume that $n\ge 2$ and $\boz=\rn$.
From Theorem~\ref{t2.2} we know that when $A\in\mathscr{A}_{\rm wusc}(\boz)$,
for locally Lipschitz functions $u$ on $\boz$ we have $(\text{Lip}_{d_A}u)^2=H(x,\nabla u)$.
In this section we show that when $A\notin\mathscr A_{\rm wusc}(\boz)$,
the (non-)coincidence  of  the above intrinsic distance and differential
structures depends on the geometry of the set where $A$ fails to be
weak-upper semicontinuous.
Indeed, we construct two examples based on a Cantor set and a Sierpinski carpet
to show that both coincidence and noncoincidence  may happen;
see, respectively,  Theorem \ref{t3.1} and Proposition \ref{p3.1}.
We consider the simple $A\notin\mathscr A_{\rm wusc}(\boz)$ defined by
 $$A_{E,\,\dz}(x)= (1-\dz 1_{E})I_n,$$
 where $\dz\in(0,\,1)$ and
 $E$ is a closed subset of $\rn$ with positive measure and empty interior.
Obviously,  $A_{E,\,\dz}$ fails to be weak upper semicontinuous at  each
$x\in E$,
and hence $A\notin\mathscr A_{\rm wusc}(\rn)$.
If $E$ is a suitable large Cantor set, then the  intrinsic distance and
differential structures never coincide.
If $E$ is a suitable large Sierpinski carpet, then  the intrinsic distance and
differential structures do coincide.
Recall that  for $\dz\in(-\fz,\,0]$,  $A_{E,\,\dz}= (1-\dz 1_{E})I_n\in\mathscr A_{\rm wusc}(\rn)$,
 and hence the associated intrinsic distance and differential structures always coincide by Theorem \ref{t2.2}.

\subsection{The large Cantor set $C_{\bf a}$}
Let ${\bf a}=\{a_j\}$ with $0<a_j<1$.
Then  the associated Cantor set $C_{\bf a}$ is constructed as follows:
 $I_{i}$,  $i=1,2$,
are the two closed intervals obtained by removing the middle open interval with length $a_1$ from $I=[0,\,1]$
and are ordered from left to right;
 when $m\ge2$,
the subintervals $I_{i_1\cdots i_m}$, $i_m=1,2$,
are the two closed intervals  obtained by removing
the middle open interval with length $a_m|I_{ i_1\cdots i_{m-1}}|$
from $I_{ i_1\cdots i_{m-1}}$,
and  are ordered from left to right; finally set
$$C_{\bf a} \equiv\bigcap_{m\in\nn}\ \bigcup_{i_1,\,\cdots,\,i_m\in\{1,2\}}I_{i_1\cdots i_m}.$$
and  $C^{(n)}_{\bf a }\equiv C_{\bf a }\times\cdots\times C_{\bf a }$.

Notice that $C^{(n)}_{\bf a}$ is closed and has empty interior,
and that $C^{(n)}_{\bf a}$ has positive $n$-dimensional Lebesgue measure
if and only if $C_{\bf a}$ has positive $1$-dimensional Lebesgue measure.
Moreover,
 $$|C_{\bf a}|=\lim_{m\to\fz}(1-a_1)\cdots (1-a_m),$$ and
by taking logarithms,  $|C_{\bf a}|>0$ if and only if ${\bf a}\in\ell^1$.
Thus,
the $n$-dimensional Lebesgue measure of $C^{(n)}_{\bf a}$ is positive  if and
only if ${\bf a}\in\ell^1$.

\begin{prop}\label{p3.1} Assume that ${\bf a}\in \ell^1$ with $0<a_j<1$ for
all $j\in\nn$.
If $\dz\in(-\fz,\,0]$, then the associated intrinsic length and differential
structure of
$A_{C^{(n)}_{\bf a},\,\dz}$ do coincide;
while if $\dz\in(0,\,1)$,  then the associated intrinsic distance and
differential structures never coincide.
\end{prop}

To prove Proposition \ref{p3.1}, we need the following (geometric) property
of $C^{(n)}_{\bf a}$ that holds even if ${\bf a}\not\in\ell^1$. To simplify
our notation, we set $d_{C^{(n)}_{\bf a},\,\dz}(x,\,y):
=d_{A_{C^{(n)}_{\bf a}},\,\dz}(x,\,y).$

\begin{lem}\label{l3.1}
Let $\dz\in(0,\,1)$.
Then for  any $x\in C^{(n)}_{\bf a}$ and any $y\in x+\rr{\bf e_1}$, we have
$d_{C^{(n)}_{\bf a},\,\dz}(x,\,y)=|x-y|$.
\end{lem}

\begin{proof}
Notice that, from the definition, we have
$d_{C_{\bf a}^{(n)},\,\dz}(x ,\,y )\ge |x-y|$
for each pair $x,y\in\mathbb{R}^n$.
Let $x\in C_{\bf a}^{(n)}$ and $y= x+r{\bf e_1}$ for some $r>0$.
It suffices to show that $d_{C_{\bf a}^{(n)},\,\dz}(x ,\,y )\le |x-y|$.

Recall that if $n=2$, it is already proved in \cite[Proposition 3.1]{kz} that
for every pair $x,\,y\in \rn$, there exists a curve $\gz$ such that except for
its endpoints,
$\gz\subset\rn\setminus C_{\bf a}^{(2)}$, and
with Euclidean length  $\ell_{\rr^2}(\gz)\le C_{C^{(2)}_{\bf a}}|x-y|$.
Hence $d_{C_{\bf a}^{(2)},\,\dz}(x,\,y)\le C_{C^{(2)}_{\bf a}}|x-y|$,
 where $C_{C_{\bf a}^{(2)}}$ is a constant determined by $C_{\bf a}^{(2)}$ and
independent of $\dz$.
 If $n\ge2$, similar arguments still apply  and hence
 for every pair
$x,\,y\in \rn$, $d_{C_{\bf a}^{(n)},\,\dz}(x,\,y)\le C_{C^{(n)}_{\bf a}}|x-y|$.
 By the construction of $C_{\bf a}$, for each $k\in\nn$ we can find
$x_k\in\rn\setminus C^{(n)}_{\bf a}$ such that
$\langle x-x_k,{\bf e_1}\rangle=0$ and
$$d_{C_{\bf a}^{(n)},\,\dz}(x,\,x_k)\le C_{C_{\bf a}^{(n)}}|x-x_k|\le \frac{1}{2k}|x-y|.$$
Take $y_k= y-(x-x_k)=x_k+r{\bf e_1}$. Then
$y_k \in \rn\setminus C^{(n)}_{\bf a}$.
Moreover, the line segment $x_k+[0,\,r]{\bf e_1}$ is contained in
$\rn\setminus C^{(n)}_{\bf a}$,
which implies that
$$d_{C_{\bf a}^{(n)},\,\dz}(x_k,\,y_k)=|x_k-y_k|=|x-y|.$$
Therefore,
$$d_{C_{\bf a}^{(n)},\,\dz}(x ,\,y )\le d_{C_{\bf a}^{(n)},\,\dz}(x_k,\,y_k)+d_{C_{\bf a}^{(n)},\,\dz}(x,\,x_k)
+d_{C_{\bf a}^{(n)},\,\dz}(y,\,y_k)\le (1+\frac1k)|x-y|,$$
which implies that
$d_{C_{\bf a}^{(n)},\,\dz}(x ,\,y )\le|x-y|$ by letting $k\to\fz$.
\end{proof}

Now we are ready to prove Proposition \ref{p3.1}.

\begin{proof}[Proof of Proposition \ref{p3.1}.]
Take $u(z)=\langle{\bf e_1}, z\rangle$ for $z\in\rn$.
For each $\dz\in(0,\,1)$ and each $x\in C^{(n)}_{\bf a}$,
by Lemma \ref{l3.1}, we have that
$$\lip_{d_{C^{(n)}_{\bf a},\,\dz}}u(x)\ge \limsup_{y\in x+\rr\bf e_1}\frac{|u(x)-u(y)|}{d_{C^{(n)}_{\bf a},\,\dz}(x,\,y)}
=
\limsup_{y\in x+\rr\bf e_1}\frac{|u(x)-u(y)|}{|x-y|}=1> {\sqrt{1-\dz}}|\nabla u(x)|,$$
which is as desired because the Cantor set $C_{\bf a}^{(n)}$ has positive
measure.
\end{proof}

From the above proof, we also conclude the following corollary.

\begin{cor}\label{c3.1}
Let $E$ be a closed subset of $\rn$. Assume that $E$ has positive measure and
empty interior.

 {\rm(}i{\rm)}  If there is a constant $C_E$ such that for each pair
$x,\,y\in \rn$
we can find a curve $\gz$ such that all of the curve except for perhaps
 countably many points lies in $\mathbb{R}^n\setminus E$ and $gz$ satisfies
$\ell_\rn(\gz)\le C_E|x-y|$,
then the intrinsic distance and differential sctructures associated to
$A_{E,\,\dz}=(1-\dz1_E)I_n$ with $\dz\in(\frac1{C_E},\,1)$,
do not coincide.

{\rm(}ii{\rm)} If there exists $\xi\in S^{n-1}$ such that for all $x\in E$,
we can find a sequence of $y\in x+\rr\xi$ satisfying $d_{E,\,\dz}(x,\,y)=|x-y|$,
then the intrinsic distance and differential structures associated to
$A_{E,\,\dz}=(1-\dz1_E)I_n$ with   $\dz\in(0,\,1)$,
do not coincide.
\end{cor}

\subsection{The large Sierpinski carpet $S_{\bf a}$}
Let ${\bf a}=\{a_j\}_{j\in\nn}$ with
$a_j\in\{\frac13,\,\frac15,\,\frac17,\cdots\}$, that is,
$a_j$ is the reciprocal
of an odd integer strictly greater than one.
A modified Sierpinski carpet $S_{\bf a}$
is constructed as follows.
First, divide the unit cube $T= [0, 1]^n$ into $a_1^{-n}$ essentially disjoint
closed congruent
subcubes, and remove the interior of the central one; denote the central one
by $T_{a_1^{-n}}$ and the others by $T_{k_1}$ with
$1\le k_1\le a_1^{-n}-1$.
When $m\ge2$, divide each $T_{k_1,\,\cdots,\,k_{m-1}}$
into $a_m^{-n}$ essentially disjoint closed congruent
subcubes and remove the interior of the central one;
denote the central one by $T_{k_1,\,\cdots,\,k_{m-1},\,a_m^{-n}}$ and the others by
$T_{k_1,\,\cdots,\,k_m}$ with $1\le k_m\le a_m^{-n}-1$.
For each $m\in\nn$, define the level $m$ precarpet by
$$S_{{\bf a},\,m}= \bigcup_{k_1}\cdots\bigcup_{k_m} T_{k_1,\,\cdots,\,k_m}.$$
The modified Sierpinski carpet $S_{\bf a }$ is defined as the limit of
precarpets $S_{{\bf a},\,m}$,
that is,  $S_{\bf a }=\cap_{m\in\nn}S_{{\bf a},\,m}.$

Obviously, $S_{\bf a }$ is closed, has empty interior, and $S_{\bf a }$ has
positive $n$-dimensional Lebesgue measure if and only if ${\bf a}\in\ell^n$.
Indeed, the $n$-dimensional Lebesgue measure of the precarpet
$S_{{\bf a},\,m}$ is $$ |S_{{\bf a},\,m}|
=( 1-a_1^{n})\cdots(1-a^{n}_{m}).$$
Thus, by taking logarithms, $|S_{\bf a}|= \lim_{m\to\fz} |S_{{\bf a},\,m}|>0$
if and only if ${\bf a}\in\ell^n$.

\begin{thm}\label{t3.1}
Let ${\bf a}=\{a_j\}_{j\in\nn}\in \ell^n$ with
$a_j\in\{\frac13,\,\frac15,\,\frac17,\cdots\}$.
Then for all $\dz \in(-\fz,\,1) $,
the associated intrinsic distance and differential structures of
$A_{S_{\bf a},\,\dz}$ do  coincide.
\end{thm}

We employ the following geometric property
of $S_{\bf a}$ to prove Theorem \ref{t3.1}. We only
need to consider the case $0<\delta<1$.

\begin{lem}\label{l3.2}
Let $\dz\in(0,\,1)$ and ${\bf a}=\{a_j\}_{j\in\nn}\in \ell^n$ with
$a_j\in\{\frac13,\,\frac15,\,\frac17,\cdots\}$.
Then there exists a subset $E\subset S_{\bf a}$ with measure zero such that
for any $\ez>0$ and each $x\in S_{\bf a}\setminus E$, we can find
$r=r(x,\,\dz,\,\ez)>0$
 which satisfies: for all $y\in B(x,\,r)$,
\begin{equation}\label{e3.x5}
d_{S_{\bf a},\,\dz}(x,\,y)\ge (1-\ez)\frac1{\sqrt{1-\dz}}|x-y|.
\end{equation}
\end{lem}

With Lemma \ref{l3.2}, we can prove Theorem \ref{t3.1} easily.
\begin{proof}[Proof of Theorem \ref{t3.1}.]

 Obviously, Lemma \ref{l2.1}  yields
$H(x,\,\nabla u(x))\le (\lip_{d_{S_{\bf a},\,\dz}}u(x))^2$ for almost all $x\in\boz$.
We now need to show that
$(\lip_{d_{S_{\bf a},\,\dz}}u(x))^2\le H(x,\,\nabla u(x))$ almost everywhere.
To this end, it suffices consider the cases $x\in\rr^n\setminus S_{\bf a}$ and
$x\in  E\subset S_{\bf a}$,
where $E$ is as in Lemma \ref{l3.2}.

\noindent {\it Case 1:} $x\in\rr^n\setminus S_{\bf a}$. It suffices to show that
if $r<d_{\rn}(x,\, S_{\bf a})/2$ and $y\in B(x,\,r)$  we have
\begin{equation}\label{e3.xx1}
d_{S_{\bf a},\,\dz}(x,\,y) =  |x-y|.
\end{equation} Indeed, \eqref{e3.xx1} will give
\[
H(x,\,\nabla u(x))=|\nabla u(x)|^2= (\lip\, u(x))^2 = (\lip_{d_{S_{\bf a},\,\dz}}u(x))^2
\]
by the definition of the pointwise Lipschitz constant.
The verification of~\eqref{e3.xx1} is done as in the proof of
Lemma~\ref{comp-dist},
with $\lambda:\mathbb{R}^n\to[1,\infty)$ given as a continuous function that
satisfies
$\lambda(y)=1$ when $y\in B(x,3r/2)$ and $\lambda(y)=1/\sqrt{1-\delta}$ when
$y\in S_{\bf a}$.

\noindent {\it Case 2:} $x\in S_{\bf a}\setminus E$.
In this case, \eqref{e3.x5} implies that
$$\limsup_{y\to x}  \frac{|y-x|}{d_{S_{\bf a}}(x,y)}\le \sqrt{1-\delta}.$$
With this, if $u$ is differentiable at $x$,   we have
\begin{align*}
 \text{Lip}_{d_{S_{\bf a}}}u(x)&=
 \limsup_{y\to x}\frac{|u(x)-u(y)|}{d_{S_{\bf a}}(x,y)}\\
 &\le \limsup_{y\to x}\frac{|u(x)-u(y)-\langle \nabla u(x),y-x\rangle|}{d_{S_{\bf a}}(x,y)}+
 \limsup_{y\to x}\frac{|\langle \nabla u(x),y-x\rangle|}{d_{S_{\bf a}}(x,y)}\\
    & \le0\cdot\limsup_{y\to x}  \frac{|y-x|}{d_{S_{\bf a}}(x,y)}+|\nabla u(x)|\limsup_{y\to x}  \frac{|y-x|}{d_{S_{\bf a}}(x,y)}\\
    &\le \sqrt{1-\delta}\, |\nabla u(x)|=\sqrt{H(x,\nabla u(x))}.
\end{align*}
This proves Theorem \ref{t3.1}.
\end{proof}

Finally, we prove Lemma \ref{l3.2}.
Notice that Lemma \ref{l3.2} is much stronger than Lemma~\ref{comp-dist};
see Remark \ref{r3.x1} below.
The proof of Lemma \ref{l3.2} relies on the following approximation of distance
established by Norris \cite{n97}. Let $\Phi\in C_c^\fz(\rn)$ be such that
$\int_\rn \Phi(x)\,dx=1$, ${\rm supp}\,\Phi\subset B(0,\,2)$ and
$0\le \Phi(x)\le1$ for all $x\in\rn$.
For $t>0$, let $\Phi_t(x)=t^{-n}\Phi(t^{-1}x)$.
Standard analysis arguments show that $\Phi_t\ast f \to f$ almost
everywhere when $f\in L^1_\loc(\rn)$.
The following lemma is due to Norris~\cite{n97}.

\begin{lem}\label{l3.3}
Let $A_{E,\,\dz,\,t}=[\Phi_t\ast (\frac{1}{1-\dz1_{E}})]^{-1}I_n$ and denote by
$d_{E,\,\dz,\,t}$ the associated intrinsic distance.
Then $d_{E,\,\dz,\,t}(x,\,y)\to d_{E,\dz}(x,\,y)$ as $t\to0$ for all
$x,\,y\in\rn$.
Moreover,
$$(d_{E,\,\dz,\,t}(x,\,y))^2=\inf_{\gz}\int_0^1\Phi_t\ast
\lf(\frac{1}{1-\dz 1_{E}}\r) (\gz(s))\lf|\frac d{ds}\gz(s)\r|^2\,ds.$$
\end{lem}

\begin{proof}[Proof of Lemma \ref{l3.2}.]
We divide the proof into 6 steps.

\noindent {\it Step 1.} Recall that for each $m$,
 $T_{k_1,\,\cdots,\,k_{m-1},\,a_m^{-n}}$ is the central cube whose interior has been
removed  from the cube $T_{k_1,\,\cdots,\,k_{m-1}}$
at step $m$
when constructing the Sierpinski Carpet $S_{\bf a}$.
For each fixed $N\in\nn\cup\{0\}$, denote by $J _{N,\,m}$ the collection of all
 $(k_1,\,\cdots,\,k_{m-1},\,j)$  such that
 $T_{k_1,\,\cdots,\,k_{m-1},\,j}$   is $N$-close to
the central cube $T_{k_1,\,\cdots,\,k_{m-1},\,a_m^{-n}}$ in the sense that
there exists a sequence $i_0,\,i_1,\,\cdots,\,i_M$ with $1\le M\le N$ such
that $i_0=j$, $i_M=a_m^{-n}$,
$i_l\ne i_s$ if $i\ne s$,
and for $0\le s<M-1$,
$$T_{k_1,\,\cdots,\,k_{m-1},\,i_{s}}\bigcap T_{k_1,\,\cdots,\,k_{m-1},\,i_{s+1}}\ne\emptyset.$$
Let $$E _{N,\,m}=\bigcup_{(k_1,\,\cdots,\,k_{m-1},\,j)\in J_{N,\,m}}T_{k_1,\,\cdots,\,k_{m-1},\,j}.$$
If $N\ge a_m^{-1}$, then $E_{N,\,m}=T_{k_1,\,\cdots,\,k_{m-1}}$.
Recall that we assume ${\bf a}\in\ell^n$.
It follows that for sufficiently large $m$ we have $N<a_m^{-1}$.
In this case, we see that
$E_{N,\,m}$ is a cube centered at the center of
$T_{k_1,\,\cdots,\,k_{m-1},a_m^{-n}}$ of edge length the
$(2N+1)\,a_m$-fraction of the edge length of the cube $T_{k_1,\,\cdots,\,k_{m-1}}.$
Observe that this
fraction tends to zero as $m\to\infty$.
We set
\[
E_N=\bigcap_{m\in\nn}\ \bigcup_{\ell\ge m}E_{N,\,\ell}=\bigcap_{\nn\ni m>N}\ \bigcup_{\nn\ni \ell\ge m}E_{N,\,\ell},
\]
Let $F$ be the union of all the $(n-1)$-dimensional faces of all the cubes
that were removed in the construction of
$S_{\bf a}$, and
\[
  E=F\cup\left(\bigcup_{N\in\nn}E_N\right).
\]
We claim that $|E|=0$.  It is easy to see that $|F|=0$.
>From the above discussion,
$$|E_{N,\,m}|\le (2N+1)^n(1-a_1^n)\cdots(1-a_{m-1}^n)a_m^n.$$
>From this, it follows that
\begin{eqnarray*}
\lf|\bigcup_{\ell\ge m}E_{N,\,\ell}\r|
   &&\le (2N+1)^n\sum_{\ell\ge m} (1-a_1^n)\cdots(1-a_{\ell-1}^n)a_\ell^n\le (2N+1)^n \sum_{\ell\ge m}  a_\ell^n,
\end{eqnarray*}
which converges to zero as $m\to\fz$ because ${\bf a}\in\ell^n$.
This implies that $E_N$ with $N\in\nn$, and hence $E$, has measure zero.

\noindent {\it Step 2.}
For any $\ez>0$, we choose $\wz N_\ez,\,N_\ez\in\nn$ such that
$$\wz N_\ez\ge \frac {100^nn^{2n}}{(1-\dz)\ez} $$
and  $N_\ez\ge(\wz N_\ez)^{n+1}$.
For each fixed $x\in S_{\bf a}\setminus E$, recall that $x\in S_{\bf a}\setminus E_N$ for all $N\in\nn$.
Since  $$S_{\bf a}\setminus E_{N_\ez}= \lf [S_{\bf a}\setminus \lf(\bigcap_{m\in\nn}\bigcup_{\ell\ge m}E_{N_\ez,\,\ell}\r)\r]
=\bigcup_{m\in\nn}  \lf[S_{\bf a}\setminus \lf(\bigcup_{\ell\ge m}E_{N_\ez,\,\ell}\r)\r]$$
there exists an  $m_x\in\nn$ such that $x\in  S_{\bf a}\setminus (\bigcup_{\ell\ge m}E_{N_\ez,\,\ell})$
for all $m\ge m_x$.  We also let $r_x>0$ be the Euclidean
distance from $x$ to the union of all removed
$T_{k_1,\,\cdots,\,k_{m-1},\,a_{m}^{-n}}$ with $m\le m_x-1$.
Since $x\not\in F$, we see that $r_x>0$.
Because ${\bf a}\in\ell^n$, we can further  find $\wz m_x\ge m_x$ such that for all
 $m\ge \wz m_x$,
 \begin{equation}\label{e3.x8}
 a_m\le (1-\dz)r_x/{2N^2_\ez}.
 \end{equation}

  {\it Step 3.}
For each $m> \wz m_x$ and $$y\in B(x,\,\wz N_\ez a_1\cdots a_m)\setminus \overline {B(x,\,\wz N_\ez a_1\cdots a_{m+1})},$$
we are going to estimate  $d_{S_{\bf a},\,\dz}(x,\,y)$ from below.
By Lemma \ref{l3.3}, we know that $d_{S_{\bf a},\,\dz,\,t}(x,\,y)$ converges to $d_{S_{\bf a},\,\dz}(x,\,y)$
as $t\to0$. So it suffices
to estimate $d_{S_{\bf a},\,\dz,\,t}(x,\,y)$ (for simplicity, we denote this quantity $d_t(x,\,y)$) from below
for all sufficiently small $t$.
By Lemma \ref{l3.3} again, we have
$$(d_t(x,\,y))^2=\inf_{\gz}\int_0^1\Phi_t\ast \lf(\frac{1}{1-\dz 1_{S_{\bf a}}}\r) (\gz(s))\lf|\frac d{ds}\gz(s)\r|^2\,ds$$
and  for each $t\in (0,\,1)$,
we can find a rectifiable curve $\gz_{\{x,\,y,\,t\}}$ joining $x$ and $y$ such that the above infimum is reached,
that is

\begin{eqnarray*}
 (d_t(x,\,y))^2 &&= \int_0^1\Phi_t\ast \lf(\frac1{1-\dz1_{S_{\bf a}}}\r)(\gz_{\{x,\,y,\,t\}}(s))\lf|\frac d{ds}\gz_{\{x,\,y,\,t\}}(s)\r|^2\,ds,
 \end{eqnarray*}
and hence,
by H\"older's inequality,
\begin{eqnarray*}
 d_t(x,\,y)
 &&\ge \int_0^1\sqrt{\Phi_t\ast\lf(\frac1{1-\dz1_{S_{\bf a}}}\r)(\gz_{\{x,\,y,\,t\}}(s)) }\lf|\frac d{ds}\gz_{\{x,\,y,\,t\}}(s)\r|\,ds.
\end{eqnarray*}
Observe that for each $t\in(0,\,1)$ and every pair $z,\,w\in\rn$,
$$ |z-w|\le d_t(z,\,w)\le \frac1{1-\dz}|z-w|,$$
which follows from $$1\le \Phi_t\ast \lf(\frac1{1-\dz1_{S_{\bf a}}}\r)\le  \frac1{1-\dz}.$$
Hence the curves $\gamma_{x,y,t}$ are Lipschitz with respect to the
Euclidean metric under a suitable parametrization,
and moreover,  with a normalization, we can assume that for the Euclidean
derivative,
$\lf|\frac d{ds}\gz_{\{x,\,y,\,t\}}(s)\r|=1$
for almost all $s\in[0,\,\ell_{\rn}(\gz_{\{x,\,y,\,t\}})]$.  Hence
  \begin{eqnarray*}
 d_t(x,\,y)  \ge
   \int_0^{\ell_{\rn}(\gz_{\{x,\,y,\,t\}})}
 \sqrt{\Phi_t\ast \lf(\frac1{1-\dz1_{S_{\bf a}}}\r)(\gz_{\{x,\,y,\,t\}}(s))}\,ds.
\end{eqnarray*}

{\it Step 4.} To estimate $d_t(x,\,y)$,
we only need to know    the length of the set
$$L_t=\lf\{s\in[0,\,\ell_\rn(\gz_{\{x,\,y,\,t\}})]:\, \Phi_t\ast \lf(\frac1{1-\dz1_{S_{\bf a}}}\r)(\gz_{\{x,\,y,\,t\}}(s))\ge (1-\ez){\frac1{1-\dz}}\r\}.$$
To this end,  observe that  if $t=a_1\cdots a_\ell$ for any large $\ell >m$
and $z\in B(x,\, \frac {\wz N_\ez}{1-\dz} a_1\cdots a_m)$ but does not belong
to the double enlargement of the (removed) cube
$T_{k_1,\,\cdots,\,k_{i-1},\,a_i^{-n}}$ with $m\le i\le \ell$, by $N_\ez\ge 4\wz N_\ez\frac1{(1-\dz)}$,
we have
\begin{equation}\label{e3.x7}
B(z,\,2t)\subset B\lf(x,\, \frac {4\wz N_\ez}{1-\dz} a_1\cdots a_m\r)\subset B(x,\, N_\ez a_1\cdots a_m).
\end{equation}
Hence
 \[
    \Phi_t\ast \lf(\frac1{1-\dz1_{S_{\bf a}}}\r)(z)
   \ge \frac1{1-\dz}\left[1- \frac{c_n}{|B(z,\,t)|}\lf|B(z,\,2t)\bigcap\lf(\bigcup_{j\ge\ell}
 T_{k_1,\,\cdots,\,k_{j-1},\,a_j^{-n}}\r) \r|\right].
\]
Note that ${\bf a}\in\ell^n$ implies $|\bigcup_{j\ge\ell}T_{k_1,\,\cdots,\,k_{j-1},\,a_j^{-n}} |\to 0$
as $\ell\to\fz$. For every $\ez>0$ there exists $\ell_0\in\nn$ which depends only on
$\ez$
such that
for all $\ell\ge\ell_0$,
$$\lf|\bigcup_{j\ge\ell}T_{k_1,\,\cdots,\,k_{j-1},\,a_j^{-n}} \r|\le \ez\, |B(z,\,t)|.$$
Therefore, if $\ell >\max\{\ell_0,\,m\}$, for the above $z$, we have
$$\Phi_t\ast \lf(\frac1{1-\dz1_{S_{\bf a}}}\r)(z)\ge \frac{1-\ez}{1-\dz}.$$

On  the other hand,
by the choice of $r_x$ and $N_\ez$ at Step 2, when
$t= a_1\cdots a_\ell\le a_\ell\le   r_x/10$
we have $\ell_{\rn}(\gz_{\{x,\,y,\,t\}})\le r_x/10$.
So  for $i\le \wz m_x-1$,
\[
  \text{dist}_{\mathbb{R}^n}(\gz_{\{x,y,t\}}, T_{k_1,\,\cdots,\,k_{i-1},\,a_i^{-n}})\ge
      \text{dist}_{\mathbb{R}^n}(x, T_{k_1,\,\cdots,\,k_{i-1},\,a_i^{-n}})\, -\, \ell_{\rn}(\gz_{\{x,\,y,\,t\}})\ge 2t.
\]
 Therefore
  $T_{k_1,\,\cdots,\,k_{i-1},\,a_i^{-n}}$  makes no contribution when we estimate
 $\Phi_t\ast \lf(\frac1{1-\dz1_{S_{\bf a}}}\r)(z)$ from below for $z\in\gz_{\{x,\,y,\,t\}}$.
This also holds for $\wz m_x\le i\le m-1$ by a similar argument.
Indeed, for $\wz m_x\le i\le m-1$,
we also have
$t= a_1\cdots a_\ell\le    a_1\cdots a_i/10$ and
$\ell_{\rn}(\gz_{\{x,\,y,\,t\}})\le  a_1\cdots a_i/10$,
and hence, in this case the  Euclidean distance from
$\gz_{\{x,\,y,\,t\}}$  to each $T_{k_1,\,\cdots,\,k_{i-1},\,a_i^{-n}}$ is at least $2t$.

Based on the above argument, the lower bound estimate of the length of $L_t$
is transferred to the upper bound estimate of the length of
$$\wz L_t=\lf\{s\in[0,\,\ell_\rn(\gz_{\{x,\,y,\,t\}})]\, :\, \gz_{\{x,\,y,\,t\}}(s)\in \bigcup_{m\le i\le \ell-1}\,
\bigcup_{k_1,\,\cdots,\,k_i}
2T_{k_1,\,\cdots,\,k_{i-1},\,a_i^{-n}}\r\}.$$
Here, $\ell$ is the positive integer such that $t=a_1\cdots a_\ell$; keep in mind that $\ell>m$.

{\it Step 5.} To estimate $\wz L_t$, we need the following key observations.

\noindent (i)  Since
  $ |x-y|\le d_t(x,\,y)\le  \frac1{1-\dz}|x-y|$,   we have
\begin{equation}\label{e3.x6}
 \gz_{\{x,\,y,\,t\}}\subset
  B\left(x,\, \frac {\wz  N_\ez}{1-\dz} a_1\cdots a_m\right).
\end{equation}
Recall that $x$ is not in any
$  N_\ez$-close cube of $T_{k_1,\,\cdots,\,k_{m-1},\,a_m^{-n}}$ whenever $m\ge m_x$.
Since $N_\ez\ge 2n\wz N_\ez\frac 1{(1-\dz)^2}$,     we have that
$$2T_{k_1,\,\cdots,\,k_{m-1},\,a_m^{-n}}\cap \gz_{\{x,\,y,\,t\}}=\emptyset.$$

\noindent (ii)
If $|x-y|\le {N_\ez}a_1\cdots a_{m+1} $, then by \eqref{e3.x8},
$|x-y|\le  a_1\cdots a_m$, and hence there are at most
  $2^n$ many  cubes $T_{k_1,\,\cdots,\,k_m}$ with $k_m<a_{m+1}^{-n}$   that
  overlaps with $\gz_{\{x,\,y,\,t\}}$, and hence,
there are at most  $2^n$ many $T_{k_1,\,\cdots,\,k_m,\,a_{m+1}^{-n}}$ such that its twice-enlargement
overlapped with $\gz_{\{x,\,y,\,t\}}$. Moreover, up to a modification of the curve
$\gz_{\{x,\,y,\,t\}}$ without increasing the $d_{t}$-length of $\gz_{\{x,\,y,\,t\}}$,
we may assume that the Euclidean length of
$\gz_{\{x,\,y,\,t\}}\cap 2T_{k_1,\,\cdots,\,k_m,\,a_{m+1}^{-n}}$ is less than $4\sqrt n a_1\cdots a_{m+1}$.
Thus by $|x-y|\ge\wz N_\ez a_1\cdots a_{m+1} $,

\begin{eqnarray*}
\wz L_{t,\,m}&&=\lf|\lf
\{s\in[0,\,\ell_\rn(\gz_{\{x,\,y,\,t\}})],\,  \gz_{\{x,\,y,\,t\}}(s)\in
\bigcup_{k_1,\,\cdots,\,k_m}
2T_{k_1,\,\cdots,\,k_{m},\,a_{m+1}^{-n}}\r\}\r|\\
&&\le 4\sqrt n 2^n a_1\cdots a_{m+1} \le   4\sqrt n 2^n\frac1{\wz N_\ez}|x-y|\le\ez|x-y|
\end{eqnarray*}

If $ {N_\ez}a_1\cdots a_{m+1}\le |x-y|\le {\wz N_\ez}a_1\cdots a_{m}$, since
$\ell_{\rn}(\gz_{\{x,\,y,\,t\}})\le \frac{\wz N_\ez}{1-\dz}a_1\cdots a_{m}$,
  there are at most $(2\frac {\wz N_\ez}{1-\dz})^n$ many $T_{k_1,\,\cdots,\,k_m}$ with $k_m<a_{m}^{-n}$
such that $$T_{k_1,\,\cdots,\,k_m}\cap \gz_{\{x,\,y,\,t\}}\ne\emptyset,$$
and hence, at most $(2\frac {\wz N_\ez}{1-\dz})^n$ many
$T_{k_1,\,\cdots,\,k_m,\,a_{m+1}^{-n}}$ such that their twice-enlargement
overlap  with $\gz_{\{x,\,y,\,t\}}$. Notice that
$4\sqrt n(2\frac {\wz N_\ez}{1-\dz})^n\le \ez N_\ez$ and
$a_{m+1}\le (1-\dz)r_x/{2N^2_\ez}\le 1/{2N^2_\ez}$. With a similar argument on $\gz_{\{x,\,y,\,t\}}$ as above,
  we have
\begin{eqnarray*}
\wz L_{t,\,m}&&=\lf|\lf
\{s\in[0,\,\ell_\rn(\gz_{\{x,\,y,\,t\}})],\,  \gz_{\{x,\,y,\,t\}}(s)\in
\bigcup_{k_1,\,\cdots,\,k_m}
2T_{k_1,\,\cdots,\,k_{m},\,a_{m+1}^{-n}}\r\}\r|\\
&&\le 2\sqrt n (2\frac N{1-\dz})^n a_1\cdots a_{m+1} \le 2\sqrt n (2\frac N{1-\dz})^n \frac1{N_\ez}|x-y|\le\ez|x-y|
\end{eqnarray*}

\noindent (iii) For each $i\ge m+1$, the numbers of $T_{k_1,\,\cdots,\,k_i,\,a_{i+1}^{-n}}$, whose  twice-enlargement
overlapped with $\gz$, is no more that   the numbers
of $T_{k_1,\,\cdots,\,k_{i}}$ with  $k_{i}< a_{i}^{-n}$ which overlaps  with $\gz$.
By induction and similar argument as above,
\begin{eqnarray*}
\wz L_{t,\,i}&&=\lf|\lf
\{s\in[0,\,\ell_\rn(\gz_{\{x,\,y,\,t\}})],\,  \gz_{\{x,\,y,\,t\}}(s)\in
\bigcup_{k_1,\,\cdots,\,k_i}
2T_{k_1,\,\cdots,\,k_{i},\,a_{i+1}^{-n}}\r\}\r|\\
&&
\le 8\sqrt na_{i+1}\lf|\lf
\{s\in[0,\,\ell_\rn(\gz_{\{x,\,y,\,t\}})],\,  \gz_{\{x,\,y,\,t\}}(s)\in
\bigcup_{k_1,\,\cdots,\,k_i}
 T_{k_1,\,\cdots,\,k_{i}}\r\}\r| \\
 && \le (8\sqrt n)^{i-m }a_{i+1}\cdots a_{m+1}\ez|x-y|.
\end{eqnarray*}

{\it Step 6.}
The three observations above yield that
$$|\wz L_t|\le\sum_{i= m}^{\ell} \wz L_{t,\,i}\le \ez|x-y|+
\sum_{i\ge m+1}(8\sqrt n)^{i-m } a_{i+1}\cdots a_{m+1}\ez|x-y|\le  3\ez|x-y|,$$
 and hence,
$$|L_t|\ge \ell_\rn(\gz_{\{x,\,y,\,t\}})-|\wz L_t|\ge (1-3\ez)|x-y|.$$
Noticing that $\ell_\rn(\gz_{\{x,\,y,\,t\}})\ge|x-y|$, we   have
$$ d_t(x,\,y) \ge \frac{\sqrt{1-\ez}}{\sqrt{1-\dz}}(1-3\ez)|x-y| \ge \frac{(1-4\ez)}{\sqrt{1-\dz}}|x-y| .$$
By the arbitrariness  of $\ez$, we conclude \eqref{e3.x5}.
\end{proof}

\begin{rem}\label{r3.x1}\rm
(i) Notice that Lemma \ref{l3.2} is much stronger than Lemma \ref{comp-dist}.
To see this, let $\lz$ be a positive continuous on $\rn$ such that
\eqref{e1.1} holds
when $A=(1-\dz1_{S_{\bf a},\,\dz})I_n$, that is,
$$\frac1{\lz(x)}|\xi|^2\le\langle A(x)\xi,\,\xi\rangle=(1-\dz1_{S_{\bf a},\,\dz})|\xi|^2\le\lz(x)|\xi|^2$$
for all $x\in\rn$ and $\xi\in\rn\setminus\{0\}$.
From  this, when $x\in S_{\bf a}$, it follows that
$\frac1{\lz(x)}\le1-\dz\le \lz(x),$
which yields  $\lz(x)\ge\frac1{1-\dz}\ge1. $
Without loss of generality, we may let that $\lz(x)=\frac1{1-\dz}$
for $x\in S_{\bf a}$.
Now fix $x\in S_{\bf a}$. Then by Lemma \ref{comp-dist},
we always have that
$$\liminf_{x\ne y\to x}\frac{d_{S_{\bf a},\,\dz}(y,x)}{|y-x|}\ge \frac{1}{\sqrt{\lambda(x)}}=\sqrt{1-\dz}, $$
which is equivalent to that for any $\ez>0$, we can find $r>0$ such that for all  $y \in  B(x,\,r)$,
\begin{equation}\label{e3.xx2}
d_{S_{\bf a},\,\dz}(x,\,y ) \ge (1-\ez)
\sqrt{1-\dz}|x-y |;
\end{equation}
and that
 $$\limsup_{x\ne y\to x}\frac{d_{S_{\bf a},\,\dz}(y,x)}{|y-x|}\le  {\sqrt{\lambda(x)}}=\frac1{\sqrt{1-\dz}},$$
which is equivalent to that for any $\ez>0$, we can find $r>0$ such that for all  $y\in  B(x,\,r)$,
\begin{equation}\label{e3.xx3}d_{S_{\bf a},\,\dz}(x,\,y ) \le (1+\ez)
\frac1{\sqrt{1-\dz}}|x-y |. \end{equation}
Obviously, we cannot obtain  \eqref{e3.x5} from \eqref{e3.xx2} and \eqref{e3.xx3},
and hence cannot obtain \eqref{e3.x5}  from Lemma \ref{comp-dist}. Indeed, \eqref{e3.x5} is much stronger than \eqref{e3.xx2}.

(ii) 
The reason why our Cantor set and Sierpinski carpet give entirely
different outcomes is the different behavior of $d_{S_{\bf a},\,\dz}$ and
$d_{C^n_{\bf a},\,\dz}$ when $\dz\in(0,\,1)$.
 Indeed, as the proof of Lemma~\ref{l3.2} shows, given almost every point
$x\in S_{\bf a}$
 and \emph{every} close-by point $y$,
 any curve that connects $x$ to $y$ with length comparable to $|x-y|$ lives in
 $S_{\bf a}$ for a significant fraction of the time, and sees $x$ as a linear
density point of
 $S_{\bf a}$. In comparison, almost every point $x$ in the Cantor set
$C^n_{\bf a}$ can be connected
 to \emph{some} near-by point by a curve of length comparable to the
Euclidean distance between
 the two points while avoiding $C^n_{\bf a}$ for almost all of the time.
\end{rem}

\begin{rem}\label{r3.1}\rm
Given an $A\in\mathscr A(\boz)$ with the intrinsic distance $d_A$, by Lemma \ref{l2.1}, for all $u\in\lip(\boz)$,
we always have $$\langle A \nabla u,\,\nabla u\rangle\le(\lip_{d_A}u)^2\le |Du|_{d_A}^2$$
almost everywhere.
If $A\in\mathscr A_{\rm wusc}(\boz)$, then the first $``\le"$ is actually $``="$,
if $A\in\mathscr A_{\rm wusc}(\boz)$ and $u\in C^1_\loc(\boz)$, the second $``\le"$ is actually $``="$.
However if  $A\notin\mathscr A_{\rm wusc}(\boz)$,
then the first $``\le"$ may be $``<"$ on some set with positive measure as shown by Proposition \ref{p3.1};
the second $``\le"$ may be $``<"$ on some set with positive measure as shown by Theorem \ref{t3.1} even for  $u\in C^1_\loc(\boz)$.
\end{rem}

\section{$L^\fz$-Variational problem for arbitrary $A\in\mathscr A(\boz)$}

In this section, we assume that $n\ge2$.
Let $A\in\mathscr A(\boz)$ and $U\Subset\boz$ be a bounded open subset.
We obtain the following existence and uniqueness of the absolute minimizer
given a boundary data
(see Section~1 for the definition of absolute minimizers).

\begin{thm}\label{t4.1}
{\rm (}i{\rm )} For every $f\in \lip(\partial U)$, there exists a unique absolutely minimizing Lipschitz extension.

{\rm (}ii{\rm )} The absolute minimizer is completely determined by the intrinsic distance in the following sense:
Let $A,\,\wz A\in\mathscr A(\boz)$ and denote by $d_A,\, d_{\wz A}$ {\rm (}resp. $H,\,\wz H${\rm )}
the corresponding intrinsic distance {\rm (}resp. Hamiltonian{\rm )}.
If \begin{equation}\label{e4.xx1}\lim_{x\ne y\to x}\frac{d_{  A}(x,\,y)}{d_{\wz A}(x,\,y)}=1 \end{equation} for almost all $x\in U$,
then $u$ is an absolute minimizer on $U$ for the Hamiltonian $H $ if and only if $u$ is
 an absolute minimizer on $U$ for the Hamiltonian $\wz H $.
\end{thm}

A special case of \eqref{e4.xx1} is that for almost every $x\in U$,
there exists $ r_x>0$ such that
$d_{  A}(x,\,y)=d_{\wz A}(x,\,y)$ for all $y\in B_d(x,\,r_x)$.

 We do not know whether if weak upper semicontinuity of $A$ could guarantee
$C^1$-regularity for the associated minimizers. However,
we have the following negative result for the $C^1$-regularity
of absolute minimizers.

\begin{prop}\label{c4.3}
Let $A= 1-\dz 1_{S_{\bf a}}$ be as in Subsection 3.2, where
${\bf a} \in \ell^n$ with
$a_j\in\{\frac13,\,\frac15,\,\frac17,\cdots\}$ and $\dz\in(0,\,1)$.
 Then there is an absolute minimizer on $U=(0,\,1)^n$ associated
to a related $L^\fz$-variational problem
that is not $C^1$-regular on $U$.
\end{prop}

Now we prove Theorem~\ref{t4.1} and Proposition~\ref{c4.3}.
Observe that the relative compactness of $U$ implies that the function $\lz$
appearing \eqref{e1.1} is bounded from above on $U$.
Without loss of generality, we may assume that $\boz=\rn$ and that
the diffusion matrix $A$ satisfies
\begin{equation}\label{e4.1}
\frac1{\lz}|\xi|^2\le \langle A(x)\xi,\,\xi\rangle\le \lz|\xi|^2
\end{equation}
for almost all $x\in\boz$ and $\xi\in\rr^n$, where $\lz\ge1$ is a fixed
constant.
Observe that $(\rn,\,d_A)$ is a length space (see for example \cite{s10}),
and hence a geodesic space due to its local compactness.
Since we have no regularity of continuity assumption on the diffusion matrix,
the approach of
using the Aronsson equations is not applicable.
Instead of this we characterize absolute minimizers via intrinsic distance;
see for example \cite{gpp,gwy,cp,dmv}.
The following Lemma \ref{l4.5} connects the absolute minimizer with a
description via pointwise Lipschitz constants;
its proof relies on (the key) Lemma~\ref{l4.6}.

\begin{lem}\label{l4.5}
Let $u\in \lip(U)$.
Then $u$ is an absolute minimizer on $U$   if
 and only if for each bounded open subset $V\Subset U$  and all
$v\in \lip(V)\cap C(\overline V)$ with
  $u|_{\partial V}= v|_{\partial V}$, one (both) of the following holds:

\noindent {\rm(}i{\rm )} $\esssup_{x\in V} \lip_{d_A}u(x) \le {\esssup}_{x\in V} \lip_{d_A}v(x);$

\noindent {\rm(}ii{\rm )} $\sup_{x\in V} \lip_{d_A}u(x) \le {\sup}_{x\in V} \lip_{d_A}v(x).$
\end{lem}

\begin{lem}\label{l4.6}
For every bounded open set $V\subset \rn$ and every $u\in\lip_\loc(\rn)$, we have
\begin{equation}\label{e4.2}
 \esssup_{x\in V}\sqrt{H(x,\,\nabla u(x))}=\esssup_{x\in V}  \lip_{d_A}u(x) = \sup_{x\in V} \lip_{d_A}u(x).
\end{equation}
  \end{lem}

\begin{proof}
By Lemma \ref{l2.1}, we always have
$$ \sqrt {F(u,\,V)}:=\esssup_{x\in V}\sqrt{H(x,\,\nabla u(x))}\le \esssup_{x\in V}  \lip_{d_A}u(x) \le\sup_{x\in V}  \lip_{d_A}u(x).$$
It then suffices to show that $ \sup_{x\in V}  \lip_{d_A}u(x)\le \sqrt {F(u,\,V)}$, which
is further reduced to showing that for every $x\in V$, there exists $r_x<d_A(x,\,\partial V)$ such that
for all $y\in B_{d_A}(x,\,r_x)$,
\begin{equation}\label{e4.xx2}
|u(x)-u(y)|\le \sqrt {F(u,\,V)}d_A(x,\,y).
\end{equation}
We divide the proof of \eqref{e4.xx2} into 4 steps.

{\it Step 1.}   Fix $x\in V$ and $0<r <\frac{1}{4}d_A(x,\,\partial V)$.   Extend $u$ from $B_{d_A}(x,\,r)$
to $\rn$ via a McShane extension as follows:
$$u_{x,\,r}(y)=\inf_{z\in B_{d_A}(x,\,r)}\{u(z)+\lip_{d_A}(u,\,B_{d_A}(x,\,r))d_A(z,\,y)\}.$$
Then $u_{x,\,r}=u$ on $B_{d_A}(x,\,r)$, and
$$H(z,\,\nabla u_{x,\,r}(z))=H(z,\,\nabla u (z))\le F(u,\,B_{d_A}(x,\,r)) \le F(u,V)$$
for almost all $z\in B_{d_A}(x,\,r)$,  and by Lemma \ref{l2.1},
$$\sqrt{H(z,\,\nabla u_{x,\,r}(z))}\le
 \lip_{d_A}u_{x,\,r}(z)\le \lip_{d_A}(u,\,B_{d_A}(x,\,r))$$
 for almost all $z\notin B_{d_A}(x,\,r)$.

 {\it Step 2.}  By the proof of Lemma~\ref{comp-dist} and the ellipticity
condition~\eqref{e4.1}, we have
 $d_A(x,y)\ge |x-y|/\sqrt{\lambda}$. Hence it follows from the last part of Step~1 above that
 for almost all $z\in\rn\setminus B_{d_A}(x,\,r)$,
 \begin{align*}
   \sqrt{H(z,\,\nabla u_{x,\,r}(z))}\le \lip_{d_A}(u,\,B_{d_A}(x,\,r))&\le \sqrt{\lambda}\, \lip(u,\,B_{d_A}(x,\,r))\\
        &= \sqrt{\lambda}\, \sup_{B_{d_A}(x,r)}|\nabla u|\le \lambda\, \sqrt{F(u,V)}.
 \end{align*}

{\it Step 3.} Now we set $$\wz A=1_{B_{d_A}(x,\,r)}\, A+
      \frac1{2\lambda} \,1_{\rn\setminus {B_{d_A}(x,\,r)}}\, A$$
 and denote by $\wz H$ and $d_{\wz A}$ the corresponding Hamiltonian and
intrinsic distance. Then for each $z\in B_{d_A}(x,\,r)$, there exists
$0<r_z<r-d_A(z,x)$ such that
whenever $d_{A}(z,\,y)<r_z$, we have \begin{equation}\label{e4.xx4}d_{\wz A}(z,\,y)=d_A(z,\,y).\end{equation}
This is seen by modifying the proof of Lemma~\ref{comp-dist} by replacing the
Euclidean
metric with the metric $d_A.$  Indeed, notice that there exists a constant
$\overline C\ge 1$ such that for all $z,\,y\in\rn$,
$$\frac1{\overline C}d_{  A}(z,\,y)\le d_{\wz A}(z,\,y)\le \overline C d_{ A}(z,\,y).$$
For $r_z< \frac1{(\overline C)^2} (r-d_A(z,x))$, we have
$$B_{d_{ A}}(z,\,r_z)\subset B_{d_{\wz A}}(z,\,\overline Cr_z)\subset B_{d_{  A}}(z,\,(\overline C)^2  r_z)\subset B_{d_{  A}}(x,\,  r).$$
Set
$$v_{z,\,r_z}(y)=\min\{d_{\wz A}(z,\,y),\,r_z\}.$$
We see that $$ H(y,\,\nabla v_{z,\,r_z}(y))=\wz H(y,\,\nabla v_{z,\,r_z}(y))\le (\lip_{d_{\wz A}}v_{z,\,r_z}(y))^2\le 1$$
for almost all $y\in B_{d_{\wz A}}(z,\,\overline Cr_z)$ and  $$ H(y,\,\nabla v_{z,\,r_z}(y))=\wz H(y,\,\nabla v_{z,\,r_z}(y))=0$$ for almost all $y\notin B_{d_{\wz A}}(z,\,\overline Cr_z)$.
Hence using $v_{z,\,r_z}$ in the definition of $d_A$,
we have $d_{A}(z,\,y)\ge d_{\wz A}(z,\,y)$ for all $y\in B_{d_{\wz A}}(z,\,\overline Cr_z)$ and hence $y\in B_{d_{ A}}(z,\,r_z)$.
Similar argument show to that $d_{A}(z,\,y)\ge d_{\wz A}(z,\,y)$ for all $y\in B_{d_{ A}}(z,\,r_z)$.
We conclude \eqref{e4.xx4} from these  inequalities.

{\it Step 4.} From the discussion in Steps~1 and~2,
$$\wz H(z,\,\nabla u_{x,\,r}(z))\le  F(u,\,V) $$ for almost every $z$.
Hence, using
$\frac1{\sqrt{F(u,\,V)}}u_{x,\,r}$ in  the definition of $ d_{\wz A}$,
we see that for all $z,\,y\in\rn$,
$$\frac1{\sqrt{F(u,\,V)}} |u_{x,\,r}(z)-u_{x,\,r}(y)|\le d_{\wz A}(z,\,y).$$
In particular, for all $z,\,y\in B_{d_A}(x,\,r)$, since $u(z)= u_{x,\,r}(z)$
and $u(y)= u_{x,\,r}(y)$,
we have
$$
|u(z)-u(y)|\le\sqrt{F(u,\,V) }d_{\wz A}(z,\,y) =\sqrt{F(u,\,V) }d_{A}(z,\,y).
$$
Applying this
with $z=x$ and $y\in B_{d_A}(x,\,r_x)$, we obtain \eqref{e4.xx2}.
\end{proof}

\begin{rem}\label{r4.7}\rm
By Lemma \ref{l2.1} above,
we always have $H(x,\,\nabla u(x))\le (\lip_{d_A}u(x))^2$ almost everywhere.
But Proposition \ref{p3.1} shows that it may happen that
$H(x,\,\nabla u(x))<(\lip_{d_A}u(x))^2$  on a set with positive measure,
and hence we can not expect $H(x,\,\nabla u(x))= (\lip_{d_A}u(x))^2$ almost
everywhere.
But as a compensation, Lemma \ref{l4.6} provides a weak variant for this:
$$\esup_{z\in B(x,\,r)}H(z,\,\nabla u(z))=\esup_{x\in B(x,\,r)} (\lip_{d_A}u(z))^2=\sup_{x\in B(x,\,r)} (\lip_{d_A}u(z))^2$$
for all $x$ and small $r$.
This phenomenon persists in the setting of general regular, strongly local Dirichlet forms; see the companion paper~\cite{ksz} of this paper.
\end{rem}

Notice that our concept of absolutely minimizing Lipschitz extension
defined in Section 1
corresponds to the strongly absolutely minimizing Lipschitz extension in \cite{jn}.
Recall that \eqref{e4.1} implies that $(\rn,\,d_A,\,dx)$ is a doubling metric measure space supporting
a $(1,1)$-Poincar\'e inequality.
 By Lemma~\ref{l4.5} and \cite[Theorem 3.1]{jn}, we have
the following existence result.

\begin{lem}\label{l4.8}
For every $f\in \lip(\partial U)$, there exists an absolutely minimizing Lipschitz extension.
\end{lem}

The uniqueness of an absolute minimizing Lipschitz extension will follow from the comparison formula.

\begin{lem}\label{l4.9}
Let $u,\,v\in \lip(U)\cap C(\overline U)$ be absolute minimizers on $U$. Then
$$\max_{x\in \overline U}[u(x)-v(x)]= \max_{x\in \partial U}[u(x)-v(x)].$$
\end{lem}

To prove Lemma \ref{l4.9}, we need the following lemmas.
First, as a consequence of Lemma~\ref{l4.5}, we have the following result.
\begin{lem}\label{l4.50}
If $u$ is an absolute minimizer on $U$, then  for all open subsets $V\Subset U$,
$\lip_{d_A}(u,V)=\lip_{d_A}(u,\partial V).$
\end{lem}

\begin{proof}
Notice that for every pair $x,\,y\in\partial V$ with $x\ne y$, by the continuity of  $d_A$
we can find $x_n,\,y_n\in V$ such that $x_n\to x$ and $y_n\to y$.  By the continuity of $u$,
  $\frac{|u(x_n)-u(y_n)|}{d_A(x_n,\,y_n)}\to\frac{|u(x)-u(y)|}{d_A(x,\,y)}$
  and hence $\lip_{d_A}(u,\,V)\ge \lip_{d_A}(u,\,\partial V)$.
Thus it suffices to prove the converse.

For $x\in\rn$, set $$w(x)=\sup_{z\in\partial V}[u(z)+\lip_{d_A}(u,\,\partial V) d_A(x,\,z)].$$ Then $\lip_{d_A}(w,\,\rn)=\lip_{d_A}(u,\,\partial V)$
and  $w=u$ on $\partial V$.
Applying Lemma \ref{l4.5}, we have $$\sup_{x\in V}\lip_{d_A} u(x)\le \sup_{x\in V}\lip_{d_A} w(x)\le
 \lip_{d_A}(u,\,\partial V).$$
Now, given a pair of points  $x,\,y\in U$, let $\gz$ be a $d_A$-geodesic curve joining $x$ and $y$.
The existence of $\gz$ is guaranteed by the fact that $(\rn,\,d_A)$
is a geodesic space~\cite{s10}.
If $\gz\subset V$, then
$$|u(x)-u(y)|\le \int_\gz\lip\, u(z)\,d_Az\le d_A(x,\,y)\,\sup_{x\in V}\lip_{d_A} u(x)
    \le  d_A(x,y)\, \lip_{d_A}(u,\,\partial V). $$
Here $d_Az$ denotes arc-length integral on $\gz$ with respect to the metric
$d_A$.
If $\gamma\not\subset V$, denote by $\hat x$ and $\hat y\in \gz\cap\partial V$
points that have shortest distance to $x$ and $y$,
respectively.
Then
\begin{eqnarray*}
|u(x)-u(y)|&&\le |u(x)-u(\hat x)|+|u(\hat x)-u(\hat y)|+|u(\hat y)-u(y)|\\
 &&\le [d_A( x,\,\hat x)+d_A( y,\,\hat y)]\sup_{x\in V}\lip_{d_A} u(x)  +d_A( \hat x,\,\hat y)\lip_{d_A}(u,\,\partial V)\\
 &&\le d_A(  x,\, y)\lip_{d_A}(u,\,\partial V).
\end{eqnarray*}
In either case, we have the inequality
\[
   \frac{|u(x)-u(y)|}{d_A(x,y)}\le \lip_{d_A}(u,\,\partial V).
\]
This means that $\lip_{d_A}(u,\,V)\le \lip_{d_A}(u,\,\partial V)$.
\end{proof}

A function $u\in C(U)$ is said to satisfy the
{\it property of comparison with cones}
 if
for all each subset $V\Subset U$, and for all $a\ge0$, $b\in\rr$ and
$x_0\in\rn\setminus V$, we have

(i) $\max_{x\in \partial V}[u(x)-C_{b,a,x_0}(x)]\le0$ implies
$\max_{x\in V}[u(x)-C_{b,a,x_0}(x)]\le0;$

(ii) $\max_{x\in \partial V}[u(x)-C_{b,-a,x_0}(x)]\ge0$ implies
$\max_{x\in V}[u(x)-C_{b,-a,x_0}(x)]\ge0,$

\noindent
where the cone function is defined by
$C_{b,a,x_0}(x)=b+a\,d_A(x,\,x_0).$
It is known that an absolute minimizer satisfies the comparison property
with cones;
see \cite{ceg} for Euclidean case and \cite{acj,jn,gwy,cp,dmv} for the
setting of metric spaces that are length spaces. For the sake of completeness,
we sketch a proof below.

\begin{lem}\label{l4.51}
An absolute minimizer satisfies the property of comparison with cones.
\end{lem}

\begin{proof}
 We prove Condition~(i), with the proof of Condition~(ii) similar (and left to the interested reader).
Let $u$ be an absolute minimizer and assume that $\max_{x\in \partial V}[u(x)-C_{b,a,x_0}(x)]\le0$.
Suppose that the condition $\max_{x\in  V}[u(x)-C_{b,a,x_0}(x)]\le0$ is \emph{not} true.
Denote by $W$  the open set of all $x\in V$ such that $u(x)>C_{b,a,x_0}(x)$.
By the above supposition, $W$ is not empty.
We can see that $u= C_{b,a,x_0}$ on $\partial W$. Since $W\subset V\Subset U$, by Corollary~\ref{l4.50}
we have
$\lip_{d_A}(u,\overline W)=\lip_{d_A}(u,W)=\lip_{d_A}(u,\partial W)=a$.
For $x\in W$, let
$\gz$ be the $d_A$-geodesic curve joining $x$ and $x_0$, and
take $z\in\partial V\cap \gz$ be a closest point to $x$.
Then $$u(x)-u(z)>C_{b,a,x_0}(x)-C_{b,a,x_0}(z)= ad_A(x_0,\,x)-a{d_A}(x_0,\,z)=ad_A(x,\,z),$$
which implies that $\lip_{d_A}(u,\,W)>a$. This is a contradiction. So $W$ must be empty.
 \end{proof}

 With the aid of Lemma~\ref{l4.5} and Lemma~\ref{l4.51},
Lemma~\ref{l4.9} will be proved by following the procedure from~\cite{as}.
Since the proof in \cite{as} is for the case $A=I_n$,
we write down the details below for the reader's convenience.
For $x\in U_r=\{z\in U:\ \overline{B_{d_A}(z,\,r)}\subset U\}$ with $r>0$,
we   set $u^r(x)= \sup_{d_A(z,\,x)\le r} u(z)$ and $u_r(x)= \inf_{d_A(z,\,x)\le r} u(z)$ and also
$$S_r^+u(x)= \frac{u^r(x)-u(x)}{r},\quad S_r^-u(x)=\frac{u(x)-u_r(x)}{r}.$$
\begin{proof}[Proof of Lemma \ref{l4.9}.]
First we claim that for $x\in U_{2r}$,
\begin{equation}\label{e4.x5}
S_r^-u^r(x)-S_r^+u^r(x)\le0\le S_r^-v_r(x)-S_r^+v_r(x).
\end{equation}
Indeed, let $y\in \overline {B_{d_A}(x,\,r)}$ and $z\in \overline {B_{d_A}(x,\,2r)}$
such that $u^r(x)=u(y)$ and $(u^{r})^r(x)=u^{2r}(x)=u(z)$. Observe
 that $(u^r)_r(x)\ge u(x)$. We then have
$$S_r^-u^r(x)-S_r^+u^r(x)=\frac1{r}[2u^r(x)-(u^r)^r(x)-(u^r)_r(x)]
\le \frac1{r}[2u (y)-u(z)-u(x)].$$
For $w\in\boz$ such that $d_A(x,\,w)=2r$, we have
$$u(w)\le u (z)=u(x)+ [u (z)-u(x)]= u(x)+\frac{[ u (z) -u(x)]}{2r}d_A(w,\,x).$$
Thus the comparison with cones property of $u$ implies that the inequality
\[
   u(w)\le u(x)+\frac{[ u (z) -u(x)]}{2r}d_A(w,\,x)
\]
holds for all
$w\in\boz$ with $d_A(x,\,w)\le 2r$. In particular,  taking $w=y$ and by $d_A(y,\,x)\le r$,
we have
$$u(y)\le  u(x)+\frac{[ u (z) -u(x)]}{2r}d_A(y,\,x)\le   u(x)+\frac12{[ u (z) -u(x)]}=\frac12{[ u (z)+u(x)]},$$
which implies the first inequality of \eqref{e4.x5}. The second inequality of \eqref{e4.x5} follows similarly.

Notice that~\eqref{e4.x5} further implies
\begin{equation}\label{e4.xx5}
    \sup_{x\in U_r}[u^r(x)-v_r(x)]= \sup_{x\in U_r\setminus U_{2r}}[u^r(x)-v_r(x)].
\end{equation}
 Given the above,   letting $r\to0$ in~\eqref{e4.xx5}, we obtain Lemma~\ref{l4.9}.  Thus it
remains only to prove~\eqref{e4.xx5}.
Assume that \eqref{e4.xx5} is {\it not} true.
Then there is some $r>0$ for which
\[
  sup_{x\in U_r}[u^r(x)-v_r(x)] > \sup_{x\in U_r\setminus U_{2r}}[u^r(x)-v_r(x)].
\]
By the continuity of $u^r -v_r $, there must exist some $y\in \overline {U}_{r}$ such that
$[u^r(y)-v_r(y)]=\sup_{x\in U_{ r}}[u^r(x)-v_r(x)]$. Because~\eqref{e4.xx5} fails, we know that $y\in U_{2r}$.
Denote by $E$ all such $y$ and set $F=\{x\in E:
\ u^r(x)=\max_{z\in E} u^r(z)\}$. Then $F$ is a closed subset of  $U_{2r}$  by the continuity of $u^r$ again.
Choose $x_0\in\partial F$. Then  because $x_0\in E$, for every $x\in U$ we have
\[
   u^r(x_0)-v_r(x_0)\ge u^r(x )-v_r(x ),
\]
and it follows from the fact that $x_0\in U_{2r}$ and for every $x\in B_{d_A}(x_0,r)$,
\[
   u^r(x_0)-v_r(x_0)\ge \inf_{z\in B_{d_A}(x_0,r)} u^r(z)-v_r(x)=(u^r)_r(x_0)-v_r(x).
\]
Taking the infimum over $x\in B_{d_A}(x_0,r)$, we obtain
$S_r^-v_r(x_0)\le S_r^-u^r(x_0)$.

{\it Case 1:} $S^+_ru^r(x_0)=0$.
Then by \eqref{e4.x5}, we have $S_r^-u^r(x_0)\le0$
and hence $S_r^-u^r(x_0)=0$,
which together with $S_r^-v_r(x_0)\le S_r^-u^r(x_0)$ implies that $S_r^-v_r(x_0)=0$. By \eqref{e4.x5} again, we have
$S_r^+v_r(x_0)\le0$ and hence $S_r^+v_r(x_0)=0$. So
$u^r$ and $v_r$ are constant on
$  B_{d_A}(x_0,\,r)$. This contradicts $x_0\in\partial F$.

{\it Case 2:} $S^+_ru^r(x_0)>0$.
Choose $z\in\overline B_{d_A}(x_0,\,r)$ such that
$rS^+_ru^r(x_0)=u^r(z)-u^r(x_0)$. Since $u^r(z)>u^r(x_0)$ and $x_0\in F$, we know that
$z\notin E$.
Therefore $u^r(x_0)-v_r(x_0)> u^r(z)-v_r(z)$, and hence
$$rS^+_rv_r(x_0)\ge v_r(z)-v_r(x_0)>u^r(z)-u^r(x_0)=rS^+_ru(x_0),$$
which together with $S_r^-v_r(x_0)\le S_r^-u^r(x_0)$  again yields that
$$S^+_rv_r(x_0)-S_r^-v_r(x_0)> S^+_ru^r(x_0)-S_r^-u^r(x_0).$$
This is a contradiction of~\eqref{e4.x5}.

So our assumption is not correct and hence \eqref{e4.xx5} is true.
\end{proof}

Now we are ready to prove Theorem \ref{t4.1}.
\begin{proof}[Proof of Theorem \ref{t4.1}]
Theorem \ref{t4.1} (i) follows from  Lemmas \ref{l4.8} and \ref{l4.9}.
Theorem \ref{t4.1} (ii) follows from Lemma \ref{l4.5} with the observation that under the assumption \eqref{e4.xx1},
$\lip_{d_A}u=\lip_{d_{ \wz A}}u$ almost everywhere for every $u\in \lip_\loc(\rn)$.
\end{proof}

Proposition \ref{c4.3} will follow from Theorem \ref{t3.1} and the following result; see \cite[Lemma 5.6]{jn} for the proof.

\begin{lem}\label{l4.52}
Let $u\in\lip (U)$ be an absolute minimizer. Then for all $x\in U$,
we have
$$\lip_{d_A}u(x)=\lim_{r\to0} S^+_ru(x)=-\lim_{r\to0} S^-_ru(x).$$
Moreover, $S^+_ru(x)$ and $S^-_ru(x)$ are increasing function with respect to
$r\in(0,\,d_A(x,\,\partial U))$.
\end{lem}

\begin{proof}[Proof of Proposition \ref{c4.3}]
We prove Proposition \ref{c4.3}
on $U=[0,1]^n\supset S_{\bf a}$.
The proof is in two steps.

\noindent{\it Step 1:}
In this step we show that each absolute minimizer on $U$ that is of class $C^1$
must satisfy $\nabla u=0$ on $S_{\bf a}$. Suppose that $u$ is an absolute minimizer on $U$ that is
of class $C^1$.
Observe that $\lip_{d_A}u$ is upper semicontinuous on $U$, that is,
for any $x\in U$, we have
$\lip_{d_A}u(x)\ge\limsup_{z\to x} \lip_{d_A}u(z).$
To see this, recall that the upper semicontinuity  is equivalent to the property that for all $L>0$,
$E_L=\{x\in U: \lip_{d_A}u(x)<L \}$ is open.
Without loss of generality, we may assume that $E_L$ is not empty. Let
$x\in E_L$.
By Lemma~\ref{l4.52}, we know that for all $z\in U$,
\begin{equation}\label{e4.xx6}
  \lip_{d_A}u(z)=\inf_{0<r<d_A(z,\,\partial U)} S^+_ru(z)=\inf_{0<r<d_A(z,\,\partial U)}\frac{u^r(z)-u(z)}{r}.
\end{equation}
It is easy to see that  for each $r>0$, $u^r$
and hence $S^+_ru$ is continuous
on $U_r=\{y\in U: d_A(y,\,\partial U)>r\}$.
Since $x\in E_L$, applying \eqref{e4.xx6}, we have $S_r^+u(x)<L$ for some $0<r<d_A(x,\,\partial U)/2$.
By the continuity of $S_r^+u$ at $x\in U_r$, we know that
there exists $r>\dz>0$ such that $B_{d_A}(x,\,\dz)\subset U_r$ and
$S_r^+u(z)<L$ for
all $z\in B_{d_A}(x,\,\dz)$. Since $d_A(z,\,\partial U)>r$, by \eqref{e4.xx6} again,
we have $\lip_{d_A}u(z)\le S_r^+u(z)<L$. This means $E_L$ is open and hence
gives the upper semicontinuity of $\lip_{d_A}u$ as desired.

Applying Theorem \ref{t3.1}, we further obtain
that $\lip_{d_A}u(x)=\sqrt{H(x,\,\nabla u(x))}$ for almost all $x\in U$,
which yields that $x\mapsto H(x,\,\nabla u(x))$ is weak upper semicontinuous on $U$.
In particular,  for almost all $x\in K$, there exists a set $E$ (which may depend  on $x$) with measure zero such that
 $$ H(x,\,\nabla u(x))\ge\limsup_{y\in U\setminus E,\ y\to x} H(y,\,\nabla u(y))
 \ge \limsup_{y\in U\setminus (S_{\bf a}\cup E),\ y\to x} H(y,\,\nabla u(y)). $$
 {Note by the construction of $S_{\bf a}$ that whenever $x\in S_{\bf a}$ and 
 $r>0$, the Lebesgue measure of $B(x,r)\setminus S_{\bf a}$ is positive. Therefore 
 $x$ is a cluster point of $\mathbb{R}^n\setminus (S_{\bf a}\cup E)$, and hence the above limit supremum 
makes sense.}  
This, together with the continuity of $\nabla u$, implies that
\begin{eqnarray*}
 (1-\dz) |\nabla u(x)|^2\ge
\limsup_{y\in U\setminus (S_{\bf a}\cup E),\ y\to x}|\nabla u(y)|^2=|\nabla u(x)|^2.
 \end{eqnarray*}
 This is a contradiction if $\nabla u(x)\ne 0$ and $\dz>0$.  Hence it follows that
 each absolute minimizer $u$ on $U$ must satisfy $\nabla u=0$ on $S_{\bf a}$ if $\nabla u$ is continuous on $U$.

\noindent{{\it Step 2:}}
Now we show the existence of an absolute minimizer $u$ on $U$
 which is either not of class $C^1$, or else
satisfies $\nabla u\ne0$ on some set $K$
of $S_{\bf a}$ with positive measure. Consider the absolute minimizer $u$ on $U$ with boundary data $f(x)=x_1$.
Assume that  $u\in C^1(U)$.
Due to the continuity of $\nabla u $, it suffices to show that
$\nabla u(x)\ne 0 $ for some $x\in S_{\bf a}$. We prove this by contradiction.
Assume that $\nabla u(x)= 0 $ for all $x\in S_{\bf a}$.
Then for any $\ez>0$, there exists $0<\ez'<\ez$ such that for all
$x\in U_\ez=\{y\in [\ez,\,1-\ez]\times[0,\,1]^{n-1}: \dist_\rn(y,\, S_{\bf a})<\ez'\}$ we have
$|\nabla u(x)|\le\ez$.
Fix $x'\in\mathbb{R}$ such that $|x'|<\ez'$, and
choose $x,\,y\in \partial U$ such that $x=(0,\,x')$, $y=(1,\,x')$
and let $\gz$ be the line segment joining $x,\,y$.
>From the construction of $S_{\bf a}$,  we see that
$\gz\cap[\ez,\,1-\ez]\times[0,\,1]^{n-1}\subset U_\ez$ when $\ez'$ is small enough and hence,
$$\int_\ez^{1-\ez}|\nabla u((\gz(t)))|\,dt\le \ez (1-2\ez).$$
Moreover, since
\begin{eqnarray*}
|\nabla u(z)|&&=\lip\, u(z)\le \frac1{\sqrt{1-\dz}} \lip_{d_{S_{\bf a},\,\dz}}(u,\,U)\\&&\le
\frac1{\sqrt{1-\dz}}\lip_{d_{S_{\bf a},\,\dz}}(f,\,\partial U)\le \frac1{ 1-\dz}\lip (f,\,\partial U)=\frac1{ 1-\dz}
\end{eqnarray*}
for all $z\in U$,
we have
$$\lf(\int_0^\ez+\int_{1-\ez}^1\r)  |\nabla u((\gz(t)))|\,dt\le 2\ez\frac1{ 1-\dz}. $$
Thus
$$1=|u(x)-u(y)|=\lf |\int_\gz (u\circ\gz)'(t)\,dt\r|\le \int_0^1|\nabla u((\gz(t)))|\,dt\le
2\ez\frac1{ 1-\dz}+\ez (1-2\ez).$$
Taking $\ez$ small enough, the term $2\ez\frac1{ 1-\dz}+\ez (1-2\ez)<1$,
which is a contradiction.
So the assumption is not true and $\nabla u(x)\ne0$ for some $x\in S_{\bf a}$,
 which contradicts Step~1 above. Therefore $u$ is not of class $C^1$ on $U$.
This proves Proposition~\ref{c4.3}.
\end{proof}

Finally, for later use, we list some more characterizations of absolute minimizers.

\begin{lem}\label{l4.53}
The following conditions on $u$ are equivalent:

\noindent(i)  $u$ is an absolute  minimizer on $U$.

\noindent(ii)  $u$ satisfies the property of comparison with cones.

\noindent(iii)  for all open sets $V\Subset U$,
$\lip_{d_A}(u,\,V)=\lip_{d_A}(u,\,\partial V)$.
\end{lem}

\begin{proof}
From Lemmas \ref{l4.50} and \ref{l4.51},
it follows that  (i)$\Rightarrow$(ii), (iii).  To obtain (ii)$\Rightarrow$(i),
we only need to notice that, with the help of Lemma \ref{l4.6},
 the argument provided by the proof of \cite[Proposition 5.8]{jn} still works
 here, without the additional weak Fubini property required in \cite{jn};
see also \cite{acj}.
  The proof of (iii)$\Rightarrow$(ii) follows directly from the proof of
Lemma~\ref{l4.51}.
\end{proof}

\section{Linear approximation when $A$ is continuous on $\boz$}\label{s5}

In this section we only consider $n\ge2$.

\begin{thm}\label{t5.1}
Let $A\in\mathscr A(\boz)$, $U\subset\boz$ and $u$ be an  absolute minimizer on $U$.
If $A$ is continuous at $x\in U$, then for every sequence $\{r_j\}_{j\in\nn}$ that converges to $0$,
there exists a subsequence ${\bf r}=\{r_{j_k}\}_{k\in\nn}$ and a vector ${\bf e}_{x,\,\bf r}$ such that
\begin{equation}\label{e2.x5}
\lim_{k\to\fz}\lf|
\frac{u(x+r_{j_k}y)-u(x)}{r_{j_k}}-\langle {\bf e}_{x,\,\bf r},\, y\rangle\r|=0
\end{equation}
and $H(x,\,{\bf e}_{x,\,\bf r})= \lip_{d_A}  u(x)$.
 Consequently, if $A $ is continuous on $U$, then \eqref{e2.x5} holds for  all $x\in U$.
\end{thm}


To prove Theorem \ref{t5.1}, we   need the following auxiliary lemmas.
We first {  look} at the case
 $x=0\in U$,  $u(0)=0$ and ${ \lip_{d_A}} u(0)\ne0$.
For any $r_0\in(0,\,d_A(0,\,\partial U))$, we know that $u$ is  an absolute minimizer on $B(0,\,r_0)\Subset U$.
 Moreover, $\nabla u\in L^\fz(B(0,\,r_0))$ and  the ellipticity function $\lz$ of~\eqref{e1.1} is
 bounded on $B(0,\,r_0)$.
 In what follows, we fix such a radius $r_0$, and without loss of generality,
  we write $U=B(0,\,r_0)$ and assume that $r_{j+1}<r_j<r_0$ for all $j$.

For each $j\in\nn$  we scale the absolute minimizer $u$ by setting $$u_j(y)=\frac{u(r_j y)}{r_j} $$
for all $y\in \frac1{r_j} U=\{\frac1{r_j}x,\,x\in U\}$.
For each $j\in\nn$, points $x\in\frac1{r_j}U$, and $\xi\in\rn$, set
{ $A_j(x)=A(r_j x)$,}
 and also set
 { $A_\infty(x)=A(0)$.
 Furthermore, for vectors $\xi\in S^{n-1}$, set}
 $H_\fz(\xi)=\langle A(0)\xi,\,\xi\rangle$.
Denote by $d_j$ the intrinsic distance of { $A_j $} for $j\in\nn\cup\{\fz\}$.

\begin{lem}
There exists $u_\fz\in W^{1,\,\fz}(\rn)$ and
 a subsequence  $\{r_{j_k}\}_{k\in\nn}$ of $\{r_j\}_{j\in\nn}$
 such that  $ u_{j_k}$  converges to
 $u_\fz $ locally uniformly   and  in weak $W^{1,\,\fz}(\rn)$.
\end{lem}

\begin{proof}
Since $\nabla u\in L^\fz(U)$ and $U=B(0,\,r_0)$ is convex, we know that
$\lip (u,\,U)\le\|\nabla u\|_{L^\fz(U)}$. Extend $u$ to $\rn$ by the McShane extension, that is, set
$$\wz u(x)=\sup_{z\in U}\{u(z)+ \lip (u,\,U)|x-z|\} $$
for all $x\in\rn$. Moreover, for each $j\in\nn$ and all $x\in\rn$, let $\wz u_j(x)=\frac{\wz u(r_jx)}{r_j}$
{ is such a McShane extension from $B(0,r_0/r_j)$ to $\mathbb{R}^n$.}

On $\rn$ we have
$\nabla \wz u_j(y)=(\nabla \wz u)(r_j y)$  and $\nabla \wz u\in L^\fz(\rn)$,
{ so it follows that} $ \wz u_j \in W^{1,\,\fz}(\rn)$ with
$\|\nabla\wz u_j \|_{L^\fz(\rn)} =\|\nabla\wz u  \|_{L^\fz(\rn)}=\|\nabla u\|_{L^\infty(B(0,r_0))}<\fz$.
Therefore, { by the Arzela-Ascoli theorem,}
there exists a subsequence $\{j_k\}_{k\in\nn}$ of $\nn$ and $u_\fz\in W^{1,\,\fz}(\rn)$
such that $ \wz u_{j_k}$ converges to $u_\fz$ locally uniformly
and {in weak $W^{1,\,\fz}(\rn)$}
 { This means that for each $(n+1)$-tuple of compactly supported continuous functions
$(\phi_0,\phi_1,\cdots,\phi_n)$, we have
\[
   \lim_k \int_{\mathbb{R}^n}\left[\wz u_{j_k}(x)\phi_0(x)+\sum_{i=1}^n\phi_i(x)\partial_i\wz u_{j_k}(x)\right]\, dx
      =\int_{\mathbb{R}^n}\left[u_\fz(x)\phi_0(x)+\sum_{i=1}^n\phi_i(x)\partial_iu_\fz(x)\right]\, dx.
\]
This weak convergence (strictly, to be called weak-* convergence), follows from the Banach-Alaouglu theorem
upon noting that $W^{1,\infty}(\mathbb{R}^n)$ is a subset of the dual of the Banach space
$(L^1(\mathbb{R}^n))^n$, together with Mazur's lemma.}

Notice that $u_j(x)=\wz u_j(x)$ whenever $x\in \frac1{r_j}U=B(0,r_0/r_j)$ for all $j\in \nn$.
Given a compact set $K$, there exists a constant $j_K$ such that   
for all  $j\ge j_K$, $K\subset\frac1{r_j}U$. Therefore $\wz u_j$ converges to $u_\fz$
on $K$ uniformly  implies that $ u_j$ converges to $u_\fz$ uniformly  as $j_K\le j\to\fz$.
\end{proof}

In what follows, for simplicity, we always write the subsequence $\{{j_k}\}_{k\in\nn}$ of
$\nn$ obtained in above Lemma 5.2 as $\nn$ by abuse of notation.

\begin{lem}\label{l5.2}
{\rm(}i{\rm)} For all $j\in\nn$, $u_j$ is an absolute minimizer on $\frac1{r_j}U$
 associated to
the Hamiltonian $H_j$ which corresponds to $A_j$.

{\rm(}ii{\rm)} If $A$ is  continuous at $0$,  then $u_\fz$ is an absolute minimizer on $\rn$
  associated to
the Hamiltonian $H_\fz$.
\end{lem}

To prove this, we need two facts given in the following; the second one relies on the  continuity of $A$ at $0$.
 We postpone the proof of Lemma~\ref{l2.x2} until after the proof of Lemma~\ref{l5.2}.

\begin{lem}\label{l2.x2}

{\rm(}i{\rm)} For $j\in\nn$ and $x,\,y\in\rn$,  $r_j d_j(x,\,y)= {d_A}(r_j x,\,r_j y)$.

{\rm(}ii{\rm)} Assume that $A$ is  continuous at $0$.
 Given a compact set $K$ and $x\in\rn$, for every $\ez>0$,  there exists $j(x,\,\ez,\,K)\in\nn$ such that
for all $j\ge j(x,\,\ez,\,K)$ and all   $y\in K$,
   $$(1-\ez)d_\fz(x,\,y)\le {d_j(x,\,y)} \le(1+\ez )d_\fz(x,\,y).$$
\end{lem}

\begin{proof}[Proof of Lemma \ref{l5.2}.]
\noindent{\it Proof of {\rm(}i{\rm)}:} Let $j\in\nn$.
It suffices to show that for all open subsets $V\Subset \frac1{r_j}U$,
$\lip_{d_j}(u_j,\,V)=\lip_{d_j}(u_j,\,\partial V)$.
By Lemma \ref{l2.x2} (i) and observing that
   $x,\,y\in V$ implies $r_jx,r_jy\in r_jV\Subset U$, we have
   $$\frac{u_j(x)-u_j(y)}{d_j(x,\,y)}= \frac{u (r_jx)-u(r_jy)}{{d_A}(r_j x,\,r_j y)} ,$$
which   yields  $\lip_{d_j}(u_j,\,V)= \lip_{d_A}(u,\, {r_j}V)$.
Similarly, $\lip_{d_j}(u_j,\,\partial V)= \lip_{d_A}(u,\, \partial({r_j}V))$
with the help of $\partial{r_j}V=r_j\partial V$.
Thus by $\lip_{d_A}(u,\, {r_j}V)=\lip_{d_A}(u,\, \partial({r_j}V))$, we obtain
 $\lip_{d_j}(u_j,\,V)= \lip_{d_j}(u_j,\,\partial V) $. { Thus the claim follows from
 Lemma~\ref{l4.53}.}

{\it Proof of {\rm(}ii{\rm)}:}
It suffices to show that $u_\fz$ satisfies the  comparison property with cones.
Let $V\Subset\rn$ and assume that for each $z\in\partial V$,
\begin{equation}\label{e5.xx1}
u_\fz(z)\le b+ad_\fz(z_0,\,z)
\end{equation}
for some $ z_0\notin V$, $b\in\rr$ and $a>0$ { which are independent of $z$.}
By Lemma \ref{l2.x2} (ii), for every $\ez>0$, there exists  $j_\ez$ such that whenever $j\ge j_\ez$ and  $z\in \overline V $,
we have
$V\Subset \frac1{r_j}U$, $$(1-\ez) d_\fz(z_0,\,z)\le  {d_j(z_0,\,z)} \le (1+\ez)d_\fz(z_0,\,z)$$
and {because $u_j\to u_\infty$ uniformly on the compact set $\overline{V}$, we also have}
$$u_\fz(z)-\ez\le u_j(z)\le u_\fz(z)+\ez.$$
Thus by \eqref{e5.xx1},
$$u_j(z)\le (b+\ez)+{ \frac{a}{(1-\ez)}}d_j(z_0,\,z)$$
for all $z\in\partial V$.
Since $u_j$ is an absolute minimizer on $\frac1{r_j}U$ associated to $H_j$ and $V\Subset \frac1{r_j}U$, we have
$$u_j(z)\le (b+\ez)+{ \frac{a}{(1-\ez)}}d_j(z_0,\,z)$$
for all $z\in V$,
which further implies that
$$u_\fz(z)\le (b+2\ez)+{ \frac{a(1+\ez)}{(1-\ez)}}d_\fz(z_0,\,z)$$ for all $z\in V$.
Due to the arbitrariness of $\ez$,
we finally have  $u_\fz(z)\le b+ad_\fz(z_0,\,z)$ for all  $z\in V$.

Similar argument also holds for $-u_\fz$. We omit the details.
So, by Lemma \ref{l4.53}, $u_\fz$ is an absolute minimizer on $\frac1{r_j}U$ associated to the Hamiltonian $H_\fz$
{ for each $r_j$, and hence on $\mathbb{R}^n$.}
\end{proof}

\begin{proof}[Proof of Lemma \ref{l2.x2}.]
(i) Let { $v$ be a locally Lipschitz function on $U$ such that
$H(x,\nabla v(x))\le 1$ for almost every $x\in U$,}
and let $v_j(z)=\frac1{r_j}v(r_jz)$.
Since $H(z,\,\nabla v(z))\le 1$ for almost all $z\in\rn$,
we have   $$ H_j(z,\,\nabla v_j(z))=    H(r_j z,\,(\nabla v)(r_jz))\le 1 $$
and hence  $$d_j(x,\,y)\ge v_j(y)-v_j(x)$$
Taking the supremum over all such $v$, we see that
$r_j d_j(x,\,y)\ge {d_A}(r_jx,\,r_jy)$. The
 inequality $r_j d_j(x,\,y)\le {d_A}(r_jx,\,r_jy)$  can be obtained similarly.

(ii) Given a compact set $K$ and $x\in\rn$, let $R>0$ and $j_x\in\nn$ be such that
for all $j\ge j_x$, we have $K\cup\{x\}\subset B_{d_j}(0,\,R)\subset B_{d_j}(0,\,2R)\subset B (0,CR)\subset \frac1{r_j}U$,
where $C>1$ is a constant depending on the lower and upper bounds of $\lz$ on $U$.
We set $v_j(z)=\min\{d_j(x,\,z),\,R\}$ for all $z\in\rn$. Note that  by Lemma \ref{l2.1},
$\langle A_j(z)\nabla v_j(z),\nabla v_j(z)\rangle\le 1$ for all $z\in\rn$, and that
$ \nabla v_j(z) =0$ for $z\notin B(0,CR)$.
Moreover, since
$A$ is continuous at $0$, for sufficiently large $j_\ez>j_x$ we have that for all $z\in B (0,CR)$,
\[
   |A_j(z)-A(0)|=|A(r_jz)-A(0)|<\ez,
\]
where we consider the operator norm on $A_j(z)-A(0)$. So for almost every $z\in B(0,CR)$,
\[
   \langle A(0)\nabla v_j(z),\nabla v_j(z)\rangle =\langle [A(0)-A_j(z)]\nabla v_j(z),\nabla v_j(z)\rangle
       +\langle\nabla A_j(z)\nabla v_j(z),\nabla v_j(z)\rangle\le L\ez+1,
\]
where  $L>0$ is a constant
related to the bound of the ellipticity function $\lambda$ on $B(0,CR)$ such that $|\nabla v|\le L$
on $B(0,CR)$. It follows that
\[
  w_j=\frac{1}{\sqrt{L\ez+1}}\, v_j
\]
can be used to compute $d_\infty$ on $B(0,CR)$. Thus  for $y\in K\subset B(0,CR)$,
\[
  d_\infty(x,y)\ge w_j(x)-w_j(y)=\frac{v_j(x)-v_j(y)}{\sqrt{L\ez+1}},
\]
that is
\[
   d_\infty(x,y)\ge \frac{d_j(x,y)}{\sqrt{L\ez+1}}.
\]

Now let $w(z)=\min\{d_\fz(x,\,z),\, R\}$ for $z\in\rn$.
An argument similar to above yields  that for $j\ge j_\ez$,
\[
   \langle A_j(z)\nabla w(z),\nabla w(z)\rangle\le L\ez+1,
\]
and so we obtain the reverse inequality
\[
  d_j(x,y)\ge \frac{d_\infty(x,y)}{\sqrt{L\ez+1}}.
\]
The conclusion of~(ii) of the lemma follows.
\end{proof}

In what follows, $S^+_r u(x) $ is as in Section 4 and
 by Lemma~\ref{l4.52}, when $u$ is an absolute minimizer
{associated to the Hamiltonian $H$ that corresponds to $A$,}  
we know that $\lip_{d_A}u(x)=\lim_{r\to 0}S^+_ru(x)$. 

\begin{lem}\label{l5.3}Assume that $A$ is  continuous at $0$. Then 

{\rm (}i{\rm )} For all $r>0$,
{ $S^+_r u_\fz(0)=\lip_{d_\fz}u_\fz(0)
   =\lip_{d_A}u(0)$  and $\sup_{x\in\rn}S^+_r u_\fz(x)\le \lip_{d_\fz}u_\fz(0)$.}

{\rm (}ii{\rm )} $\sup_{x\in\rn}\lip_{d_\fz}u_\fz(x)=\lip_{d_\fz}(u_\fz,\,\rn)= \lip_{d_\fz}u_\fz(0).$

\end{lem}

\begin{proof}
{ By Lemma~\ref{l2.x2}, we know that $u_\fz$ is an absolute minimizer associated with
$H_\fz$. Hence by Lemma~\ref{l4.52} and the claim~(i) of this lemma, the claim~(ii) will follow.
Hence it suffices to prove the claim~(i).}
We first observe that
\begin{equation}\label{e5.x1}
\lip_{d_\fz}(u_\fz,\,\rn)\le { \lip_{d_A}u(0)}.
\end{equation}
Indeed, for all $x,\,y\in\rn$ with $x\ne y$, by $r_j d_j(x,\,y)= {d_A}(r_j x,\,r_j y)$, we have
\begin{eqnarray}\label{eq:A}
 \frac{|u_\fz(x)-u_\fz(y)| }{d_\fz(x,\,y)}
 &&=\lim_{j\to\fz}\frac{| u_j(x)-u_j(y) |}{d_j(x,\,y)}
 = \lim_{j\to\fz}\frac{ |u (r_j x)-u (r_j y)|}{{d_A}(r_j x,\,r_j y)}.
\end{eqnarray}
letting $\gz$ be the geodesic
 curve { in the metric $d_j$ (and hence in the metric $d_A$)}
 joining $r_jx,\,r_jy$,  (such $\gz$ exists when $j$ large enough because then $r_jx,r_jy\in B(0,R)$),
we obtain
$$|u (r_j x)-u (r_j y)|\le \int_\gz \lip_{d_A} u(z)|dz|\le\sup_{z\in\gz}{ \lip_{d_A}}u(z)\, {d_A}(r_jx,\,r_jy).$$
Thus
 $$ \frac{|u_\fz(x)-u_\fz(y)|}{d_\fz(x,\,y)} \le\lim_{j\to\fz}\sup_{z\in B_{d_A}(0,\,{d_A}(r_j x,\,r_j y))}{ \lip_{d_A}u(z).}$$
Observing that { $\lip_{d_A}u$} is upper semicontinuous (for details see  the proof of Proposition \ref{c4.3},
{ in particular,~\eqref{e4.xx6}}),
and by ${d_A}(r_jx,\,r_jy)\to0$ as $j\to\fz$,
we arrive at
$$ \frac{|u_\fz(x)-u_\fz(y)|}{d_\fz(x,\,y)} { \le \lip_{d_A}} u(0).$$
{ This proves~\eqref{e5.x1}.}

 From \eqref{eq:A} and Lemma \ref{l4.52}, it also follows that
${ \lip_{d_\fz}} u_\fz(x)\le S^+_r u_\fz(x)\le { \lip_{d_A}}u( 0)$ for all $x\in\rn$ and $r>0$.
Moreover, we {will show below that}    
\begin{equation}\label{e5.xx9}    
S^+_ru_\fz(0){ \ge \lip_{d_A}}u(0).    
\end{equation}    
{ From this, and applying the above discussion to $x=0$, by Lemma \ref{l4.52} we have}
$$S^+_ru_\fz(0)={ \lip_{d_\fz}} u_\fz(0)= { \lip_{d_A}}u(0)$$ for all $r>0$.
Since we already have
{ $\lip_{d_\fz}u_\fz(x)\le \lip_{d_\fz}(u_\fz,\mathbb{R}^n)$ for all $x\in\mathbb{R}^n$,}
we obtain (i).
This proves  {
Lemma \ref{l5.3}. 
}
Hence we end the proof of Lemma~\ref{l5.3} by establishing \eqref{e5.xx9}.

 By the continuity of { $u_\fz$,} for every $\ez>0$ there exists $0<\dz_0<1/4$ such that
 whenever $\dz\in(0,\,\dz_0)$,
\begin{equation}\label{e2.x6}
\frac{\sup_{d_\fz(0,\,y)\le r}u_\fz(y)}{r}\ge \frac{\sup_{d_\fz(0,\,y)\le (1+\dz)r}u_\fz(y)}{(1+\dz)r} -\ez
  = S^+_{(1+\dz)r}u_\fz(0)-\ez;
\end{equation}
 and since $u_j\to u_\fz$ locally uniformly as $j\to\fz$, there exists  $j_\ez\in\nn$
 such that
 for all $j\ge j_\ez$ and $y\in   B_{d_\fz}(0,\,2r)\setminus B_{d_\fz}(0,\,r/2)$,
\begin{equation}\label{e5.xx6}
 \frac{ u_\fz(y)-u_\fz(0)}{d_\fz(0,\,y)}
 \ge \frac{u_j(y)-u_j(0)}{ d_j(0,\,y)} -\ez
  = \frac{u(r_jy)-u (0)}{{d_A}(0,\,r_j y)}-\ez .
\end{equation}
Moreover, by Lemma 5.4 (ii), for any $\dz\in(0,\,\dz_0)$, there exists $j_\dz$ such that
 such that for all $j\ge j_\dz$ and $y\in B_{d_\fz}(0,\,2r)$,
\begin{equation}\label{e5.xx7}(1-\dz){d_j}(0,\,y)\le d_\fz(0,\,y) \le (1+\dz){d_j}(0,\,y).
\end{equation}

Let $z_j\in\rn$ such that  $d_A(0,\, z_j)\le r_jr$ and
\begin{equation}\label{e5.xx8}\frac{u(z_j)-u (0)}{r_jr}=\frac{\max_{{d_A}(z,\,0)\le r_jr}u(z)-u(0)}{r_jr}
=S^+_{r_jr}u(0)\ge { \lip_{d_A}}u(0).\end{equation}
By {the comparison Lemma~\ref{l4.9},} we know that $d_A(0,\, z_j)= r_jr$.
Let $y_j=z_j/r_j$.
Observe that  whenever $j\ge j_\dz$, \eqref{e5.xx7} implies that
$$ d_\fz(0,\, y_j)\le (1+\dz)d_j(0,\,  y_j)\le (1+\dz)\frac1{r_j}d_A(0,\,  z_j)=(1+\dz) r$$
and similarly,  $(1-\dz)r\le d_\fz(0,\,y_j).$
By this, the increasing property of $S^+_ru_\fz(0)$
with respect to $r$ given by Lemma \ref{l4.52}, \eqref{e5.xx6}, Lemma 5.4 (i) and \eqref{e5.xx8}, we have
whenever $j\ge \max\{j_\dz,\,j_\ez\}$,
\begin{eqnarray*}
 S^+_{(1+\dz)r}u_\fz(0)&&\ge S^+_{d_\fz(0,\,y_j)}u_\fz(0)
 \ge
  \frac{ u_\fz(y_j)-u_\fz(0)}{d_\fz(0,\,y_j)}\\
  &&
    \ge
  \frac{u_j(y_j)-u_j(0)}{{d_j}(0,\,y_j  )}-\ez
=
  \frac{u(z_j)-u (0)}{{d_A}(0,\,z_j  )}-\ez\ge { \lip_{d_A}}u(0)-\ez,
\end{eqnarray*}
which together with \eqref{e2.x6}
implies $S^+_{ r}u_\fz(0)\ge { \lip_{d_A}}u(0)-2\ez.$ From this, we conclude that
$S^+_ru_\fz(0)\ge S^+u(0)$, and hence \eqref{e5.xx9}.
  \end{proof}

\begin{lem}\label{l5.4}
Assume that $A$ is continuous at $0$.
There exists  ${\bf e}\in\rn$ such that
$u_\fz(x)=\langle {\bf e},\,x\rangle$
for all $x\in\rn$ and { $H_\fz(x,{\bf e})=\lip_{d_\fz}u_\fz(0)$.}
\end{lem}

\begin{proof}
Notice that by Lemma \ref{l5.2},
$u_\fz$ is an absolute minimizer on $\rn$ associated to the Hamiltonian $H_\fz$.
Moreover $u_\fz$ satisfies Lemma \ref{l5.3} (i) and (ii).
If $A(0)=I_n$, then $H_\fz(\xi)=\langle \xi,\,\xi\rangle$ and hence Lemma \ref{l5.4}
follows from  \cite{ceg}. Generally, Lemma \ref{l5.4} follows from Lemma 3.4 of \cite{wy},
where  $H_\fz$ satisfies the conditions required there.
\end{proof}

\begin{proof}[Proof of Theorem \ref{t5.1}]
Without loss of generality, we may assume that $x=0$, $u(x)=0$ and $\lip_{d_A}u(0)>0$.
Indeed, set $\wz u(z)=u(x+z)-u(x)$ for $z\in\rn$, and $\wz H(z,\,\xi)=\langle A(x+z)\xi,\,\xi\rangle$.
Then $\wz u$ is an absolute minimizer of $\wz H$ if and only if $ u$ is an absolute minimizer of $ H$;
 Theorem \ref{t5.1} holds for $\wz u$ at $0$ if and only if Theorem \ref{t5.1} holds for $  u$ at $x$.
But $\wz u(0)=0$.
We also notice that  $\lip_{d_A}u(0)=0$  together with the equivalence
of $d$ and the Euclidean distance yields that  $\lip\, u(0)=0$, and hence, $u$ is differentiable at $0$ with $\nabla u(0)=0$.
This  means that \eqref{e2.x5} holds with ${\bf e}=0$
and $  \lip_{d_A} u(0)=H(0,\,{\bf e})=0.$

Now we consider the scaling of the absolute minimizer $u$ by $u_j(y)=\frac{u(r_j y)}{r_j} $ as above.
$u_\fz$ is the limit of some subsequence  of $u_j$, which is still denoted by $u_j$ for simple.
Then \eqref{e2.x5} is reduced to showing
$u_\fz(z)=\langle {\bf e},\,z\rangle$ for some vector ${\bf e}\in\rn$ and
$H_\fz({\bf e})=\lip_{d_A}u(0)$. But this follows from Lemma \ref{l5.4} and Lemma \ref{l5.3}.
\end{proof}

\begin{rem}\rm
(i) In the above proof, we do need the  continuity of $A$ at $x$
to conclude that $d_\fz$ is the intrinsic distance associated to $H_\fz$.
We do not know what happens if $A$ is only assumed to be weak upper
semicontinuous at $x.$

(ii)
We expect that the above linear approximation property provided by
Theorem  \ref{t5.1}  may help to understand
the $C^1$-regularity or the differentiability everywhere of the
absolute minimizer
associated to a continuous diffusion matrix $A$, see
\cite{s05,es,es12,wy} in the case $A=I_n$.

\end{rem}

\noindent Pekka Koskela

\noindent Department of Mathematics and Statistics,
P. O. Box 35 (MaD),
FI-40014, University of Jyv\"askyl\"a,
Finland
\smallskip

\noindent{\it E-mail }:   \texttt{pkoskela@maths.jyu.fi}

\bigskip

\noindent Nageswari Shanmugalingam   

\noindent Department of Mathematical Sciences, P.O.Box 210025, University of Cincinnati,
 Cincinnati, OH 45221-0025, U.S.A.

\smallskip

\noindent {\it E-mail}: \texttt{shanmun@uc.edu}  

\bigskip

\noindent Yuan Zhou

\noindent
Department of Mathematics, Beijing University of Aeronautics and Astronautics, Beijing 100191, P. R. China

\noindent{\it E-mail }:  \texttt{yuanzhou@buaa.edu.cn}



\begin{thebibliography}{99}


\bibitem{as}
S. N. Armstrong  and C. K. Smart, An easy proof of Jensen's theorem on the
uniqueness of infinity harmonic functions, Calc. Var. Partial Differential
Equations 37 (2010), 381-384.

\vspace{-0.3cm}
\bibitem{a1}
G. Aronsson,
Minimization problems for the functional $\sup_x F(x, f(x), f' (x))$,
Ark. Mat. 6 (1965), 33-53.

 \vspace{-0.3cm}
\bibitem{a2}
 G. Aronsson, Minimization problems for the functional $\sup_x F(x, f(x), f' (x))$. II,
Ark. Mat. 6 (1966), 409-431.

\vspace{-0.3cm}
\bibitem{a3}
G.  Aronsson,  Extension of functions satisfying Lipschitz conditions,
Ark. Mat. 6 (1967), 551-561.

 \vspace{-0.3cm}
\bibitem{a4}
  G. Aronsson,  Minimization problems for the functional $\sup_x F(x, f(x), f' (x))$. III, Ark. Mat. 7 (1969),
 509-512.




\vspace{-0.3cm}
 \bibitem{acj}
G. Aronsson, M. G. Crandall and P. Juutinen,   A tour of the theory of
absolutely minimizing functions, Bull.
Amer. Math. Soc. (N. S.) 41 (2004), 439-505.



\vspace{-0.3cm}
\bibitem{cp}
T. Champion and L. Pascale, Principle of comparison with distance functions
for absolute minimzer,
J. Convex Anal. 14 (2007), 515-541.

\vspace{-0.3cm}
\bibitem{cpp}
T. Champion, L. De Pascale and F. Prinari, $\Gamma$-Convergence and absolute
minimizers for supremal
functionals, ESAIM Control Optim. Calc. Var. 10 (2004), 14-27.

\vspace{-0.3cm}
\bibitem{c99}
J. Cheeger, Differentiability of Lipschitz functions on metric measure spaces,
 Geom. Funct. Anal.  9  (1999),  428-517.



 \vspace{-0.3cm}
 \bibitem{ceg}
M. G. Crandall, L. C. Evans and R. F.Gariepy,  Optimal Lipschitz extensions
and the infinity Laplacian. Calc.
Var. Partial Differential Equations 13 (2001), 123-139.


 \vspace{-0.3cm}
 \bibitem{dmv}
F. Dragoni, J. J. Manfredi and D. Vittone,
Weak Fubini property and infinity
harmonic functions in Riemannian and
sub-Riemannian manifolds, Trans. Amer. Math. Soc. to appear.

\vspace{-0.3cm}
\bibitem{es}
L. C. Evans and O. Savin, $C^{1,\az}$ regularity for infinity harmonic
functions in two dimensions,
Calc. Var. Partial Differential Equations 32  (2008), 325-347.

\vspace{-0.3cm}
\bibitem{es12}
L. C. Evans and C. K. Smart,
Everywhere differentiability of infinity harmonic functions,
 Calc. Var. Partial Differential Equations, to appear.

\vspace{-0.3cm}
\bibitem{fot}
M. Fukushima,  Y. $\rm \bar O$shima and M. Takeda, Dirichlet forms and
symmetric Markov processes,
de Gruyter Studies in Mathematics, 19. Walter de Gruyter \& Co., Berlin,
1994. 392 pp.



 \vspace{-0.3cm}
 \bibitem{gpp} A. Garroni, M. Ponsiglione
 and F. Prinari, From 1-homogeneous supremum functional to difference
quotients: relaxation and $\Gamma$-convergence,
 Calc. Var. Partial Differential Equations
 27 (2006), 397-420.


 \vspace{-0.3cm}
 \bibitem{gwy}
 R. Gariepy, C. Wang and Y. Yu,
 Generalized cone comparison principle for viscosity solutions of the
 Aronsson equation and absolute minimizers,
 Comm. Partial Differential Equations 31  (2006),  1027-1046.

 \vspace{-0.3cm}
 \bibitem{j93}
R. Jensen,   Uniqueness of Lipschitz extensions: minimizing the sup norm of
the gradient, Arch. Ration.
Mech. Anal. 123 (1993), 51-74.

 \vspace{-0.3cm}
 \bibitem{ju}
P. Juutinen, Absolutely minimizing Lipschitz extensions on a metric space,
Ann. Acad. Sci. Fenn. Math. 27 (2002),  57-67.

 \vspace{-0.3cm}
 \bibitem{jn}
P. Juutinen and N. Shanmugalingam,  Equivalence of AMLE, strong AMLE, and
comparison with cones in metric measure space,
Math. Nachr. 279 (2006), 1083-1098.



 \vspace{-0.3cm}
 \bibitem{kz}
 P. Koskela and Y. Zhou, Geometry and analysis of Dirichlet forms, submitted.

 \vspace{-0.3cm}
 \bibitem{ksz}
 P. Koskela, N. Shanmugalingam and Y. Zhou, $L^\infty$-variational problem on
metric measure spaces, preprint.



\vspace{-0.3cm}
\bibitem{n97}
J. R. Norris, Heat kernel asymptotics and the distance function in Lipschitz
Riemannian manifolds,
Acta Math. 179 (1997),  79-103.

\vspace{-0.3cm}
\bibitem{pssw}
Y. Peres, O. Schramm, S. Sheffield and D. B. Wilson, Tug-of-war and the
infinity Laplacian,
J. Amer. Math. Soc. 22 (2009), 167-210.


\vspace{-0.3cm}
\bibitem{s05}
O. Savin, $C^1$ regularity for infinity harmonic functions in two dimensions,
Arch. Ration. Mech. Anal.  176  (2005), 351-361.




\vspace{-0.3cm}
\bibitem{s10}
P. Stollmann, A dual characterization of length spaces with application to Dirichlet metric spaces,
Studia Math.  198  (2010),  221-233.


\vspace{-0.3cm}
\bibitem{s94}
K.-T. Sturm,
Analysis on local Dirichlet spaces. I. Recurrence, conservativeness and
$L^p$-Liouville properties,
J. Reine Angew. Math.  456  (1994), 173-196.


\vspace{-0.3cm}
\bibitem{s97} K.-T. Sturm,
Is a diffusion process determined by its intrinsic metric?
Chaos Solitons Fractals 8 (1997),  1855-1860.

\vspace{-0.3cm}
\bibitem{wy}
C. Wang and Y. Yu,  $C^1$-regularity of the Aronsson equation in
${\mathbf R}^2$,
 Ann. Inst. H. Poincar\'e Anal. Non Lin\'eaire 25 (2008),  659-678.

\end{thebibliography}
\end{document}